\documentclass[12pt]{article} 
\usepackage{amsmath, amsthm, bm}
\usepackage[right=2.5cm, top=2.5cm, left=2.5cm, bottom=2.5cm, includefoot]{geometry}
\usepackage{tikz} 
\usepackage{amssymb}
\usepackage{float}
\usepackage{graphicx}
\usepackage{setspace}
\usepackage{authblk}
\usepackage{bbm,multirow}
\usepackage{hyperref}
\usepackage{capt-of}
\usepackage[utf8]{inputenc} 

\usepackage{geometry} 
\usepackage{graphicx} 

\usepackage{booktabs} 
\usepackage{array} 
\usepackage{paralist} 
\usepackage{verbatim} 
\usepackage{subfig} 

\usepackage{fancyhdr} 
\usepackage{sectsty}
\allsectionsfont{\sffamily\mdseries\upshape} 

\renewenvironment{thebibliography}[1]{
  \begin{oldthebibliography}{#1}
    \setlength{\itemsep}{0em}
    \setlength{\parskip}{-1em}
}
{
  \end{oldthebibliography}
}
\usepackage[nottoc,notlof,notlot]{tocbibind} 
\usepackage[titles,subfigure]{tocloft} 

\theoremstyle{definition}
\newtheorem{example}{Example}[section]
\newtheorem{remark}{Remark}[section]
\newtheorem{defn}{Definition}[section]
\newtheorem{prop}{Proposition}[section]
\newtheorem{cor}{Corollary}[section]
\newtheorem{thm}{Theorem}[section]
\newtheorem{lemma}{Lemma}[section]

\title{Efficient computation of complementary set partitions, with applications to an extension and estimation of generalized cumulants}
\author{Elvira Di Nardo \thanks{Department of Mathematics, University of Turin, Turin, Italy, email: elvira.dinardo@unito.it (corresponding author)}, Giuseppe Guarino \thanks{Local Health Authority of Potenza, Potenza, Italy, email: giuseppe.guarino@webmail.it}}

\date{} 
\begin{document}
\maketitle
\begin{abstract}
This paper develops new combinatorial approaches to analyze and compute special set partitions, called complementary set partitions, which are fundamental in the study of generalized cumulants. Moving away from traditional graph-based and algebraic methods, a simple and fast algorithm is proposed to list complementary set partitions  based on two-block partitions, making the computation more accessible and implementable also in non-symbolic programming languages like R. Computational comparisons in {\tt Maple} demonstrate the efficiency of the proposal.
Additionally  the notion of generalized cumulant is extended using multiset subdivisions and multi-index partitions  to include scenarios with repeated variables and to address more sophisticated dependence structures.  A formula is provided that expresses generalized multivariate cumulants as linear combinations of multivariate cumulants, weighted by coefficients that admit a natural combinatorial interpretation. Finally, the introduction of dummy variables and specialized multi-index partitions enables an efficient procedure for estimating generalized multivariate cumulants with a substantial reduction in data power sums involved.
\end{abstract}
\section{Introduction}
Generalized cumulants were introduced by McCullagh in \cite{mccullagh1984tensor} and permits to compute joint cumulants of polynomials in a random sample, such as linear, quadratic forms and so on \cite{Kollo}. To give a simple example, suppose we need the covariance of $\sum_i a_i X_i$ and $\sum_{j,k} a_{jk} X_j X_k,$ where  ${\bm X}=(X_1, \ldots, X_n)$ is a  vector of random variables (r.v.'s) not necessarily independent and identically distributed (i.i.d.). 
By the multilinear property of covariance we have
$$
{\rm cov}\bigg(\sum_i a_i X_i, \sum_{j,k} a_{jk} X_j X_k \bigg) =\sum_{i,k,j} 
a_i a_{jk} \kappa_{i,jk} 
$$
where $\kappa_{i,jk} = {\rm cov}(X_i, X_j X_k)$ is the generalized cumulant of $X_i$ and $X_j X_k.$ These calculations find applications in various fields, for example sample cumulants, Edgeworth series, conditional cumulants, Bartlett's identities, see  \cite{mccullagh1984tensor} and \cite{stafford1994automating}.  Many other examples can be found in \cite{andrews2000symbolic}, \cite{barndorff1989asymptotic} and \cite{mccullagh2018tensor}, to which we refer the reader interested in such applications.

Although generalized cumulants are very useful, their practical application is constrained by considerable computational complexity. This stems from their dependence on a specific class of set partitions, known as complementary set partitions. Manual listing of these partitions becomes increasingly difficult as the number of indexes increases, even if they are of moderate size. The development of automatic tools for these calculations would therefore be of great benefit.

To the best of our knowledge, there are currently no available procedures for generating complementary set partitions in non-symbolic programming environments, like {\tt R} \cite{Rmanual}.  Computer algebra tools have been implemented by Wang, Andrews and Stafford \cite{stafford1994symbolic}, McCullagh and Wilks\cite{mccullagh1988complementary}, and  Kendall \cite{kendall1994computer}.  Nevertheless, due to the inefficiency of the resulting computation times, these methods have not been adopted in practice. Consequently, the tables provided in \cite{mccullagh2018tensor} continue to serve as the standard reference, even though they are incomplete, as a full listing would require an impractical number of pages.  As a result, any missing information must be manually computed, utilizing the fundamental properties of complementary set partition (see Section 2 for a short review). In \cite{stafford1994automating}, Stafford suggested a strategic change by avoiding the combinatorial complexity of the problem. Instead of initially computing complementary set partitions and then generalized cumulants, he reversed the process, proposing that generalized cumulants be calculated first as a way to recover complementary set partitions. This strategy was implemented \footnote{J. E. Stafford kindly provided us with routines in Mathematica.} in {\tt Mathematica 3.0} \cite{ram1996}
(further details can be found in \cite{andrews2000symbolic}). Despite relying exclusively on algebraic calculations, the automatic procedure’s implementation in non-open source software such as {\tt Mathematica} prevented its widespread adoption. Moreover, this approach would be inefficient when applied to other formulas involving complementary set partitions, like those for polykay cumulants. Consequently, the tables in \cite{mccullagh2018tensor} are still the primary reference tool for complementary set partitions and related calculations  \cite{mesters2024non}. 

The purpose of this paper is twofold, both computational and theoretical. 
As suggested in \cite{mccullagh2018tensor}, complementary set partitions can be computed using connected graphs. Here, we propose an alternative and novel combinatorial method. For a partition with $m$ blocks, we only use the two-block partitions of the set $[m]=\{1, \ldots,m\}$  to recover all partitions that are not complementary to the given one. Comparative analysis of computational times with both connected graph-based algorithms and other strategies demonstrates the efficiency of our approach, which is simple enough to be implemented in non-symbolic, open-source software\footnote{All routines are available upon request} like {\tt R}. From a theoretical perspective, the main contribution of this paper is the extension of the notion of generalized cumulants to r.v.'s indexed by multiset subdivisions \cite{di2008unifying}. Generalized cumulants are usually expressed in terms of joint cumulants products indexed by partition blocks, with r.v.'s having distinct subscripts. By extending the indexing to multiset subdivisions, we allow repetitions and thus r.v.'s powers. In this framework, we define generalized multivariate cumulants as intermediate quantities between multivariate cumulants and moments, and we provide a closed-form expression to represent these quantities in terms of products of multivariate cumulants. The key tool relies on defining a labelling rule to distinguish repeating r.v.'s   and on characterizing complementary set partitions using appropriate vector subspaces. These subspaces, which are generated by binary vectors encoding the partition blocks, allow for an efficient transformation of the auxiliary distinct r.v.'s back to the original set. Furthermore, we provide a combinatorial interpretation of the coefficients multiplying the products of multivariate cumulants in the expression of generalized multivariate cumulants, which are all 1 in the classical case of generalized cumulants

As an application, an unbiased estimator for generalized multivariate cumulants is proposed, based on multivariate polykays \cite{direview}. Traditionally, computing multivariate polykays involves set partitions and computer algebra tools \cite{di2009new}. By employing multi-index partitions, these formulas have been implemented efficiently even in non-symbolic programming environments like {\tt R} \cite{di2022kstatistics}. In this context, the use of the labeling rule and dummy variables further speed up the computation of these estimators by reducing the number of multivariate polykays involved. The main idea is to reduce the estimation of a generalized multivariate cumulant to the estimation of a generalized cumulant of the same order, but involving fewer power sum symmetric functions. In turn, the estimation of a generalized cumulant is further simplified to the estimation of a joint cumulant, thus decreasing overall computational time.

The paper is organized as follows. Section 2 provides an overview of complementary set partitions and their properties.  Section 3 briefly discusses the existing literature on methods for generating complementary set partitions and introduces the new combinatorial approach proposed in this paper.  In Section 4, generalized multivariate cumulants and labeling rules are introduced. In the same section we define the  $\vec{{\bm 1}}_n$-partitions and show how the subspaces spanned by their columns  characterize complementary set partitions.  A closed-form formula is then provided to express generalized multivariate cumulants in terms of products of multivariate cumulants, using an appropriate transformation of multi-index partitions  into $\vec{{\bm 1}}_n$-partitions. These products are multiplied by specific coefficients, which are given a combinatorial interpretation.  Section 5 focuses on practical applications. We provide computational comparisons between the combinatorial method introduced here and other existing methods for generating complementary set partitions. Next, we explain how this same approach can be used for the efficient computation of generalized multivariate cumulants. We then illustrate the application of the labeling rule and suitable dummy variables to achieve efficient computation of generalized multivariate cumulant estimators. Conclusions and some open problems end the paper.
\section{Background on complementary set partitions}
Recall that a partition $\pi$ of $[n]=\{1,\ldots,n\}$  is a set with $m \leq n$ nonempty subsets of $[n],$ named blocks of the partition, such that every integer is in exactly one subset. The set of partitions of $[n]$ is usually denoted by $\Pi_n,$ to which we refer in the following unless otherwise specified, and  the subset of partitions in $m$ blocks is denoted by $\Pi_{n,m} \subseteq \Pi_n.$ To specify the blocks  $\{B_i\}$ of $\pi \in \Pi_{n,m},$ we write  $\pi=B_1 | \ldots |B_m.$ In particular the trivial partition is denoted by ${\mathbf 1}_n=1\ldots n$ and the singleton partition is denoted by  ${\mathbf 0}_n =1|\ldots|n.$ The refinement is the natural partial order among set partitions, that is $\pi \leq \tilde{\pi}$ ($\pi$ refines $\tilde{\pi}$)  if $\pi = \tilde{\pi}$ or every block of $\pi$ is a subset of some block of $\tilde{\pi}.$ In particular, the least upper bound $\pi \lor \tilde{\pi}$ is the finest partition which is refined by both $\pi$ and $\tilde{\pi}.$ 
\begin{defn}  \cite{mccullagh1988complementary}
Two set partitions  $\pi$ and $\tilde{\pi}$ are said to be complementary if their least upper bound  is the trivial partition ${\mathbf 1}_n,$ that is $\pi \lor \tilde{\pi} = {\mathbf 1}_n.$ 
\end{defn}
\vspace*{-0.6cm}
\subparagraph{Swapping property.}
Let us begin by recalling the canonical representation ({\it cr1}) of a partition 
$\pi \in \Pi_{n,m}$ as described in \cite{mccullagh1988complementary}: within each block, elements appear in increasing order; the blocks themselves are arranged in decreasing order of cardinality, and blocks of the same size are ordered lexicographically. In such a way, the block cardinalities of  $\pi \in \Pi_{n,m}$  gives a partition\footnote{ Recall that  $\lambda \vdash n$ denotes a partition of an integer $n,$ that is a sequence $\lambda = (\lambda_1, \lambda_2, \ldots , \lambda_m),$ with $m \leq n,$ where $\lambda_j$ are decreasing  integers, named parts of $\lambda$ and such that $\sum_{j=1}^m \lambda_j = n.$ } $\lambda \vdash n$ in $m$ parts. Note that if $\pi, \tilde{\pi} \in \Pi_{n,m}$ share the same integer partition  
of block cardinalities, then $\pi$ is obtained from $\tilde{\pi}$ by swapping some integers among blocks, that is there exists a permutation that takes the canonical representation of  $\pi$  into the canonical representation of $\tilde{\pi}$.  Thus, the complementary set partitions of $\tilde{\pi}$ can be recovered from those of $\pi$ swapping the corresponding interchanged integers in the blocks (swapping property).
\begin{example} \label{swapex}
The complementary set partitions of $\pi=1|234$ are
$12|3|4, 13|2|4, 14|2|3,$ $123|4, 124|3, 12|3\,4, 134|2, 13|24, 14|23, 1234.$ The complementary set partitions of $\tilde{\pi}=123|4$ are the complementary set partitions of $\pi$ with $1$ swapped with $4,$ that is
$1|2 4|3, 1|2|3 4,$ $14|2|3, 1|234, 124|3, 13|24, 134|2,12|34, 14|23, 1234.$
\end{example}
\vspace*{-0.4cm}
\subparagraph{Intersection matrices.} Complementary set partitions can be grouped into equivalence classes using intersection matrices.
\begin{defn}\label{intmat}
If $k,l \leq n, \pi_1=B_1|  \ldots, B_k \in \Pi_{n,k}$ and $\pi_2=C_1|  \ldots| C_l \in \Pi_{n,l},$  the intersection matrix $M_{[k \times l]},$ denoted by $\pi_1 \cap \pi_2,$ is such that its $(i,j)$-th element  $M_{ij}$ is the cardinality of  $B_i \cap C_j,$ that is  $M_{ij}= |B_i \cap C_j|.$ 
\end{defn}
Given $\pi$, the partitions $\pi_1$ and $\pi_2$ are equivalent if they have equal intersection matrices, that is if $\pi \cap \pi_1 = \pi \cap \pi_2,$ after suitably permuting the blocks of $\pi_1$ and $\pi_2$ or the columns of $\pi \cap \pi_1$ and $\pi \cap \pi_2.$ 
\vspace*{-0.3cm}
\subparagraph{Generalized cumulants.} 
Suppose to have ${\bm X}=(X_1, X_2, \ldots)$ r.v.'s not necessarily i.i.d. 
\begin{defn} \label{defgencum1} \cite{brillinger2001time}
Let ${\mathfrak I}$ be a set of $n$ subscripts chosen among the ones of $X_i$'s and $\pi = B_1| \ldots |B_m$ a partition in $m$ blocks of ${\mathfrak I}.$ Generalized cumulants are the joint cumulants
\begin{equation}
\mathfrak{K}(\pi) = \mathfrak{K}^{B_1| \ldots| B_m} = {\mathcal K} \left( \prod_{i \in B_1} X_i, \ldots,  \prod_{i \in B_m} X_i\right).
\label{gencum}
\end{equation}
\end{defn}
In the literature, to simplify the notation,  the subscripts  in  ${\mathfrak I}$ are labeled by increasing integer numbers, that is ${\mathfrak I}=\{j,\ldots,l\}$ is replaced by $\{1,\ldots,n\}=[n].$ Thus $n$ is said the degree of the generalized cumulant while its order is the number $m$ of blocks.  If $m=1,$ \eqref{gencum} reduces to the joint moment
\begin{equation}
\mathfrak{K}({\mathbf 1}_n) = \mathfrak{K}^{[n]} = {\mathcal K} \left(X_1 \cdots X_n\right) = {\mathbb E}[X_1 \cdots X_n].
\label{jointM}
\end{equation}
If $m = n > 1,$ \eqref{gencum} reduces to the joint cumulant 
\begin{equation}
\mathfrak{K}({\mathbf 0}_n) = \mathfrak{K}^{B_1| \ldots| B_m} = {\mathcal K} \left( X_1, \ldots,  X_n\right).
\label{jointcum}
\end{equation}
The expression of generalized cumulants in terms of joint cumulants is \cite{mccullagh1988complementary}  
\begin{equation}
\mathfrak{K}(\pi) = \sum_{\tilde{\pi}: \pi \lor \tilde{\pi} = {\bm 1}_n} \prod_{B_i \in \tilde{\pi}} \kappa(B_i)
\label{gencumincumul}
\end{equation}
where the summation runs over all $\tilde{\pi}$ that are complementary set partitions of $\pi$ and $\kappa(B_i)$ denotes the joint cumulant of the r. v.'s corresponding to the indices in $B_i.$
\section{Methods to list complementary set partitions}
We begin this section with a short review of existing methods in the literature for  listing complementary set partitions of $\pi \in \Pi_n.$ In the second part, we propose a new method relied on two-blocks partitions which offers significant computational advantages (see Section \ref{Appl}). 
\vspace*{-0.3cm}
\subparagraph{Connected graphs.}
 Given a partition $\pi \in \Pi_n,$ suppose ${\mathcal G}(\pi)$ the graph whose vertices are labeled by $1,\ldots,n$ and edges connect vertices in the same block partition.  Thus ${\mathcal G}(\pi)$ is partitioned into cliques\footnote{A clique in an undirected graph is  a set of vertices in which every pair of distinct vertices is adjacent, forming a complete induced subgraph. A graph is complete if every vertex is directly connected to all other vertices by unique edges. A graph is connected if there exists a path, i.e. a sequence of edges, linking every pair of vertices.}, each one corresponding to a block of $\pi$. Among these graphs ${\mathcal G}({\mathbf 1}_n)$  is uniquely both complete and connected, consisting of a single clique. 
\begin{defn}\label{def1}
The sum of two graphs ${\mathcal G}(\pi)$ and ${\mathcal G}(\tilde{\pi})$ is the graph ${\mathcal G}(\pi) \oplus {\mathcal G}(\tilde{\pi})$ having the same $n$ vertices and edges obtained by the union of the edges of ${\mathcal G}(\pi)$ and those of ${\mathcal G}(\tilde{\pi}).$  
\end{defn}
\begin{thm}\label{thm31}  \cite{mccullagh1984tensor}
$\pi, \tilde{\pi} \in \Pi_n$ are complementary if and only if ${\mathcal G}(\pi) \oplus {\mathcal G}(\tilde{\pi})$ is connected.  
\end{thm} 
 According to Theorem \ref{thm31}, a way to check if two partitions $\pi $ and $\tilde{\pi}$  are complementary is to verify that there exists a path in ${\mathcal G}(\pi) \oplus {\mathcal G}(\tilde{\pi})$ for each pair of vertices. To this aim, a classical strategy relies on working with its Laplacian matrix\footnote{The Laplacian matrix of ${\mathcal G}(\pi)$ is defined as $L_{i,j}[{\mathcal G}(\pi)] =  \deg(i)$ if $i = j,$ with $\deg(i)$ the number of edges terminating at the vertex $i,$ $L_{i,j}[{\mathcal G}(\pi)] =-1$  if $i \ne  j$ but $i$ and $j$ are connected by an edge, otherwise $0.$}.
If $L[{\mathcal G}(\pi) \oplus {\mathcal G}(\tilde{\pi})]$ has rank $n-1,$ then ${\mathcal G}(\pi) \oplus {\mathcal G}(\tilde{\pi})$ is connected as the number of connected components is equal to $1$ \cite{chung1997spectral}. In summary, 
to find all complementary set partitions of $\pi$,  the algorithm checks each $\tilde{\pi} \in \Pi_n$ by computing the rank of $L[{\mathcal G}(\pi) \oplus {\mathcal G}(\tilde{\pi})].$  In Table 1, this algorithm is listed as {\tt csp\_LA}. For completeness, we also employed the {\tt IsConnected} function from the {\tt Hypergraphs Maple 2024} package to verify graph connectivity. In Table 1,  this algorithm is listed as {\tt csp\_GR}. 
\subparagraph{Stafford's algorithm.}
The following example shows  the technique proposed by Stafford \cite{stafford1994automating} for obtaining complementary set partitions.  In Table 1, this algorithm is listed as {\tt cps\_JS}. 
\begin{example} \label{exampleStafford}
To express ${\rm cov}(X_1, X_2 X_3)$ in terms of joint cumulants, first set
$Y_1 = X_1$ and $Y_2 = X_2 X_3.$ Then 
from ${\rm cov}(Y_1, Y_2) = {\mathbb E}[Y_1, Y_2] - {\mathbb E}[Y_1]  {\mathbb E}[ Y_2]$ recover
${\rm cov}(X_1, X_2 X_3) = {\mathbb E}[X_1 X_2 X_3] - {\mathbb E}[X_1]{\mathbb E}[X_2 X_3].$ Now express the joint moments of $X_i$'s in terms of joint cumulants, that is 
${\mathbb E}[X_1]  = \kappa_{1}, {\mathbb E}[X_2 X_3]  =  \kappa_{2} \kappa_{3}  + \kappa_{2 \,3}$ and
 ${\mathbb E}[X_1 X_2 X_3] =  \kappa_{1} \kappa_{2} \kappa_{3} + \kappa_{1} \kappa_{2\,3} + \kappa_{2} \kappa_{1\,3}  + \kappa_{3} \kappa_{1\,2}  + \kappa_{1\,2\,3}.$ Plug these expressions in $ {\mathbb E}[X_1 X_2 X_3] - {\mathbb E}[X_1]{\mathbb E}[X_2 X_3]$ to get ${\rm cov}(X_1, X_2 X_3) = \kappa_{123} + \kappa_{2} \kappa_{1\,3} + \kappa_{3} \kappa_{1 \, 2}.$
Thus the complementary set partitions of $1|23$ are $123, 13|2, 12|3.$
\end{example}
\subsection{The new two-blocks partition method} \label{newal}
Suppose $\pi = B_1| \ldots| B_m \in \Pi_{n,m}$ with $m \geq 2.$ The steps of the new proposed method (see Table 1, algorithm {\tt csp}) to list the complementary set partitions of $\pi$ can be summarized as follows:
\begin{description}
\item[{\it a)}]  let  $\{C_1, C_2\}$ be a partition in two blocks of $[m]$ and set  
\begin{equation}
A_1 = \cup_{i \in C_1} B_i \qquad \hbox{and} \quad A_2 = \cup_{i \in C_2} B_i;
\label{(7)}
\end{equation}
\item[{\it b)}] for all $\tilde{\pi}_1 \in \Pi_{A_1}$ and $\tilde{\pi}_2 \in \Pi_{A_2}$
consider the partition $\tilde{\pi} = \tilde{\pi}_1 \cup \tilde{\pi}_2 \in \Pi_n;$  
\item[{\it c)}] repeat steps {\it a)} and {\it b)} for each partition in two blocks of $[m]$ and denote with ${\mathcal T}^{(n)}_{\pi,m}$ the set of all partitions $\tilde{\pi}$ constructed as above; 
\item[{\it d)}]  the set ${\Pi}_n -  {\mathcal T}^{(n)}_{\pi,m},$ is of  all complementary set partitions of $\pi.$   
\end{description}

\begin{example}
Consider $\pi=B_1|B_2|B_3 \in \Pi_{5,3}$ with $B_1=\{1\}, B_2 = \{2,3\}, B_3 = \{4,5\}.$ The partitions $C_1|C_2$ in two blocks of the indexes $\{1,2,3\}$ are $1|23, 3|12$ and  $12|3.$ Set $A_1  = B_1 \cup B_2 = \{1,2,3\}$ and $A_2=B_3=\{4,5\}.$ The partitions  
of $A_1$ are $ 123, 1|23, 13|2, 12|3, 1|2|3,$ those of $A_2$ are $45, 4|5.$ Thus the partitions $\tilde{\pi}$ in step {\it b)} are $123|45, 123|4|5,$  $1|23|45, 1|23 |4|5, 13|2|45, 13|2|4|5, 
12|3|45, 12|3|4|5,  1|2|3|45, 1|2|3|4|5.$ The above are partitions of $\Pi_5$ not complementary to  $1|23|45$ as their least upper bound with $\pi$ is $\pi$ itself or $A_1| A_2 = 123|45.$  Repeating the same arguments for $A_1  = B_1, A_2=B_2 \cup B_3$ and $A_1  = B_1 \cup B_3, A_2=B_2,$   all set partitions of $\Pi_5$ not complementary to $1|23|45$ are retrieved. 
\end{example}
The following theorem proves that the steps {\it a)}-{\it c)} of the procedure outlined above indeed generate all set partitions that are not complementary to $\pi \in \Pi_{n,m}.$  
\begin{thm} \label{thm3.2}
If $\pi=B_1| \ldots|B_m \in \Pi_{n,m},$ with $m \geq 2,$ then the set of all partitions not complementary to $\pi$ is
\begin{equation}
{\mathcal T}^{(n)}_{\pi,m} =\bigg\{\tilde{\pi}=\tilde{\pi}_1 \cup
\tilde{\pi}_2 \in \Pi_n| \tilde{\pi}_1 \in \Pi_{A_1}, \tilde{\pi}_2 \in \Pi_{A_2}\,  \hbox{\rm with $A_1, A_2$ in \eqref{(7)} at varying}   \, \, C_1|C_2 \in \Pi_{m,2} \bigg\}.
\label{fund3}
\end{equation}
\end{thm}
\begin{proof}
We prove that ${\mathcal T}^{(n)}_{\pi,m}$ contains exactly those partitions not complementary to $\pi.$ Indeed, if $\tilde{\pi} \in {\mathcal T}^{(n)}_{\pi,m}$ then $\pi \lor \tilde{\pi} \leq A_1|A_2$ since $\pi \leq A_1| A_2$ and $\tilde{\pi} \leq A_1|A_2,$ as from \eqref{(7)} $\tilde{\pi}_1 \in \Pi_{A_1}, \tilde{\pi}_2 \in \Pi_{A_2}$ respectively. Vice-versa, suppose $\tilde{\pi}$ not complementary to 
$\pi.$ Then there exists $\pi^{\star} \in \Pi_n$ such that $\pi \lor \tilde{\pi} = \pi^{\star} \ne  {\bm 1}_n.$ If $\pi^{\star} \in \Pi_{n,m}$ with $m \geq 2,$ as
$\pi \leq \pi^{\star},$ there exists $A_1|A_2 \in \Pi_{n,2}$ such that  $\pi^{\star} \leq A_1|A_2$ since it's sufficient to join two or more block of $\pi$ until the structure of two blocks is  reached. 
\end{proof}
\begin{remark} \label{rem3.1}
Note that $|\Pi_n| = B_n$ the $n$-th Bell number. Therefore the number of pairs $(\tilde{\pi}_1, \tilde{\pi}_2)$ in \eqref{fund3} is $B_{|A_1|} \times B_{|A_2|}$  with
$|A_j| = \sum_{i \in C_j} |B_i|, j=1,2.$ Among these pairs, there are some that provide the same partition $\tilde{\pi}.$ Therefore, the number of not complementary set partitions of $\pi$ can be obtained using the inclusion-exclusion principle. Indeed set $l=|\Pi_{m,2}|=2^{(m-1)}-1$ and denote with
\begin{equation}
{\mathcal T}_i =\bigg\{\tilde{\pi}=\tilde{\pi}_1 \cup
\tilde{\pi}_2 \in \Pi_n| \tilde{\pi}_1 \in \Pi_{A_1}, \tilde{\pi}_2 \in \Pi_{A_2}\,  \hbox{\rm with $A_1, A_2$ in \eqref{(7)} for a fixed}   \, \, C_1|C_2 \in \Pi_{m,2} \bigg\}
\label{fund3bis}
\end{equation}
for $i=1,\ldots,l.$ Thus we have
\begin{equation}
|{\mathcal T}^{(n)}_{\pi,m}| = | \cup_{i=1}^l {\mathcal T}_i| = \sum_{k=1}^l 
(-1)^{k+1}  \left( \sum_{1 \leq i_1 \leq \cdots \leq i_k \leq l} |{\mathcal T}_{i_1} \cap \ldots \cap {\mathcal T}_{i_k}| \right).
\label{fund3ter}
\end{equation}
\end{remark}
\section{Generalized multivariate cumulants}
This section introduces an extension of generalized cumulants through the use of multiset subdivisions of the subscripts of $X_i$'s. Roughly speaking a multiset is a set in which repetitions are allowed. In particular, a multiset $M$ is a pair $(\bar{M}, f)$ where $f: \bar{M} \mapsto {\mathbb N}$ and $\bar{M}$ is a set called the support of $M$  \cite{di2008unifying}. For every $j \in \bar{M}, f(j)$ is said the multiplicity of $j$ and  $|M| = \sum_{j \in \bar{M}} f(j)$ is the lenght of $M.$ A multivariate cumulant can be recovered from a joint cumulant using a  multiset of subscripts as follows
\begin{equation}
\kappa_{{\bm i}}({\bm X}) = \kappa_{i_1,\ldots,i_n}(X_1, \ldots,X_n) = {\mathcal K}(\underbrace{X_1, \ldots, X_1}_{i_1}, \ldots, \underbrace{X_n, \ldots, X_n}_{i_n}) 
\label{multisetcum10}
\end{equation} 
where ${\mathcal K}$ denotes the joint cumulant  and the order of the repeated r.v.'s does not matter. The \lq\lq bag\rq\rq $M$ of the subscripts in \eqref{multisetcum10} is an example of multiset with  support $\bar{M}=[n]$
and multiplicities $i_1, \ldots, i_n.$ Unless stated otherwise, we will write $M = \{1^{(i_1)}, ,\ldots,n^{(i_n)} \}$ in the following sections.  When one or more integers $i_1, \ldots, i_n$ equals  $0$, the corresponding r.v.'s are omitted from the bag.   For example $\kappa_{1,0,2}(X_1, X_2, X_3)= {\mathcal K}(X_1, X_3, X_3).$ 

Just as a set can be partitioned into disjointed subsets, similarly a multiset can be subdivided into disjointed submultisets. A multiset $M_i = (\bar{M}_i,f_i)$ is a submultiset of $M= (\bar{M},f)$ if  $\bar{M}_i \subseteq \bar{M}$ and $f_i(j) \leq f(j)$ for all $j \in \bar{M}.$  A subdivision\footnote{Formally, a subdivision of $M= (\bar{M},f)$ is a multiset of $m \leq |M|$ non-empty submultisets $M_i = (\bar{M}_i,f_i)$ of $M$ such that $\cup_{i=1}^m \bar{M}_i = \bar{M}$ and $\sum_{i=1}^m f_i(j)=f(j)$ for all $j \in \bar{M}.$  } of a multiset $M$ is a multiset of submultisets of $M$, such that their disjoint union returns $M.$ Thus the following extends Definition \ref{defgencum1}.
\begin{defn} \label{defgencum0}
Let $M$ be a multiset of $|M|$ subscripts chosen among the ones of 
$X_i$'s and $S$ a subdivision in submultisets $M_i=(\bar{M}_i, f_i), i=1, \ldots,m \leq |M|.$ A generalized multivariate cumulant is the joint cumulant  
\begin{equation}
\mathfrak{K}(S) = \mathfrak{K}^{M_1|\ldots|M_m} = {\mathcal K} \bigg( \prod_{i \in \bar{M}_1} X_i^{f_1(i)}, \ldots, \prod_{i \in \bar{M}_m} X_i^{f_m(i)}\bigg).
\label{gencum3}
\end{equation}
\end{defn}
To simplify the notation, the multiset $M$ of subscripts is denoted  by $\{1^{(i_1)}, \ldots,n^{(i_n)}\}$ where each $j \in [n]$ corresponds to  $i_j$ repeated subscripts of 
$X_i$'s.  If $M = [n]$ with no repeating integers, then $S$ returns a set partition $\pi$ and \eqref{gencum3} reduces to \eqref{gencum}. If $m=1$ then $M_1 = M$ and  $\mathfrak{K}^{M}={\mathbb E}[X_1^{i_1}\cdots X_n^{i_n}]$ a multivariate moment. If $m=|M|$ then \eqref{gencum3} returns the multivariate cumulant \eqref{multisetcum10}.  

To handle multivariate cumulants, multi-index partitions can be used instead of multiset subdivisions.  Multi-index partitions play a key role in the multivariate Faà di Bruno formula and are better suited for non-symbolic computational environments \cite{di2022kstatistics}. Recall that a composition of a multi-index\footnote{The multi-indexes  $ (0,0,\ldots,0),(1,1,\ldots,1) \in {\mathbb N}_0^n$ will be denoted by $\vec{{\bm 0}}_n$ and  $\vec{{\bm 1}}_n$ respectively.} ${\bm i}=(i_1,  \ldots,i_n) \in {\mathbb N}_0^n, $  in $m$ multi-indexes is a matrix $\Lambda = ({\bm \lambda}_1, \ldots, {\bm \lambda}_m)$ such that ${\bm \lambda}_1 + \cdots + {\bm \lambda}_m = {\bm i},$ see \cite{di2011new}. The length of ${\Lambda}$ is the number of its columns, denoted by $l({\Lambda}).$  A multi-index partition is a composition whose columns are in lexicografic order. In the following, we fix a reverse lexicografic order\footnote{As example $(a_1, b_1) > (a_2, b_2)$ if $a_1 > a_2$ or $a_1 = a_2$ and $b_1 > b_2.$}. Usually multi-index partitions are denoted by  $\Lambda = (\boldsymbol{\lambda}_1^{r_1}, \boldsymbol{\lambda}_2^{r_2}, \ldots) \vdash {\bm i}$ with  $r_1 \geq 1$ columns equal to $\boldsymbol{\lambda}_1 > \boldsymbol{\lambda}_2, r_2 \geq 1$ columns  equal to $\boldsymbol{\lambda}_2 > \boldsymbol{\lambda}_3$ and so on. In such a case $l(\Lambda) = r_1  + r_2 + \cdots.$ 
\begin{remark} \label{faa} If ${\bm X}=(X_1,  \ldots, X_n)$ is a random vector, then $\mu_{\bm i}({\bm X})=\mu_{i_1 \ldots i_n}({\bm X})={\mathbb E}[{\bm X}^{\bm i}] ={\mathbb E}[ X_1^{i_1} \cdots X_n^{i_n}] $ denotes its  multivariate moment of order ${\bm i}$ while its multivariate cumulants  $\kappa_{\bm i}({\bm X})=\kappa_{{i_1 \ldots i_n}}({\bm X})$ are defined by $\kappa_{\bm 0} = 0$  and the identity $
\sum_{|{\bm i}| \geq 1} \kappa_{\bm i}({\bm X}) \frac{{\bm z}^{\bm i}}{{\bm i}!}
= \ln \left( 1 + \sum_{|{\bm i}| \geq 1} {\mathbb E}[{\bm X}^{\bm i}] \frac{{\bm z}^{\bm i}}{{\bm i}!}\right)$
where ${\bm z}^{\bm i} = z_1^{i_1} \cdots z_n^{i_n},  {\bm i}! = i_1! \cdots i_n!, |{\bm i}| = i_1 + \cdots + i_n.$  Formulae giving multivariate moments in terms of multivariate cumulants (and viceversa) using multi-index partitions are \cite{di2022kstatistics} 
\begin{equation}
\mu_{\bm i}({\bm X}) = \sum_{\Lambda \vdash {\bm i}} d_ {\Lambda} \kappa_{\bm X}(\Lambda) \quad \hbox{\rm and} \quad \kappa_{\bm i}({\bm X}) = \sum_{\Lambda \vdash {\bm i}} (-1)^{l(\Lambda) - 1} (l(\Lambda)-1)! d_ {\Lambda}  \mu_{\bm X}(\Lambda)
\label{cummom}
\end{equation}
where the sum is over all multi-index partitions $\Lambda= (\boldsymbol{\lambda}_1^{r_1}, \boldsymbol{\lambda}_2^{r_2}, \ldots) \vdash {\bm i}$ of length $l(\Lambda)$ and 
 \begin{equation}
d_{\Lambda}  := {\bm i}! \prod_i \frac{1}{({\bm \lambda}_i!)^{r_i}  r_i! }, \quad \mu_{\bm X}(\Lambda) := \prod_i  [\mu_{\boldsymbol{\lambda}_i}({\bm X})]^{r_i}, 
\quad \kappa_{\bm X}(\Lambda) := \prod_i  [\kappa_{\boldsymbol{\lambda}_i}({\bm X})]^{r_i}.
\label{multfunct}
\end{equation}
If $\Lambda  \vdash \vec{{\bm 1}}_n$ then $d_{\Lambda}=1$ and \eqref{cummom} gives
\begin{equation}
\mu_{\scriptscriptstyle{11\ldots1}}({\bm X}) = \sum_{\Lambda \vdash \vec{{\bm 1}}_n} 
\kappa_{\bm X}(\Lambda) \quad \hbox{\rm and} \quad  \kappa_{\scriptscriptstyle{11\ldots1}}({\bm X}) = \sum_{\Lambda \vdash \vec{{\bm 1}}_n} (-1)^{l(\Lambda) - 1} (l(\Lambda)-1)! \mu_{\bm X}(\Lambda)
\label{cummom1}
\end{equation}
with  $\mu_{\scriptscriptstyle{11\ldots1}}({\bm X})$ and $\kappa_{\scriptscriptstyle{11\ldots1}}({\bm X})$  joint moment and joint cumulant  respectively.
\end{remark}
Subdivisions of $M = \{1^{(i_1)}, ,\ldots,n^{(i_n)} \}$ are in one-to-one correspondence with multi-index partitions  of ${\bm i}=(i_1, \ldots, i_n)$ \cite{di2022kstatistics}.
\begin{example} \label{6.5}
Consider the multiset $M=\{1^{(1)},2^{(2)},3^{(2)} \}.$
Examples of subdivisions are $S_1=1|23|23$ and  $S_2=123|23$ corresponding to the multi-index partitions $\Lambda_1=({\bm \lambda}_1, {\bm \lambda}_2^2), \Lambda_2=({\bm \lambda}_3,{\bm \lambda}_4) \vdash (1,2,2)^{\intercal}$ respectively, with
${\bm \lambda}_1 = (1, 0, 0)^{\!\intercal}, {\bm \lambda}_2 = (0, 1, 1)^{\!\intercal}, {\bm \lambda}_3 = (1, 1, 1)^{\!\intercal}, {\bm \lambda}_4 =  (0, 1, 1)^{\!\intercal}.$
\end{example}
With this notation, the definition of generalized multivariate cumulants can be re-expressed in terms of multi-index partitions rather than multiset subdivisions.
\begin{defn} \label{defgencum}
The generalized multivariate cumulant of degree $n$ and order $l(\Lambda) = r_1 + \cdots + r_m$ corresponding to the multi-index partition  $\Lambda =({\bm \lambda}_1^{r_1}, \ldots, {\bm \lambda}_m^{r_m}) \vdash {\bm i} \in {\mathbb N}^n$  is the joint cumulant
\begin{equation}
{\mathfrak K}_{\Lambda}({\bm X}) ={\mathfrak K}^{r_1| \ldots| r_m}_{{\bm \lambda}^{\!\intercal}_1; \ldots; {\bm \lambda}^{\!\intercal}_m}({\bm X}) = {\mathcal K} \big( \underbrace{ {\bm X}^{{\bm \lambda}_1}, \ldots,{\bm X}^{{\bm \lambda}_1}}_{r_1}, \ldots, \underbrace{{\bm X}^{{\bm \lambda}_m}, \ldots, {\bm X}^{{\bm \lambda}_m}}_{r_m} \big).
\label{gencummultp}
\end{equation}
\end{defn}
\noindent 
If $m=1$ and $r_1=1,$ then $\Lambda = {\bm i}$ and \eqref{gencummultp} returns the ${\bm i}$-th multivariate moment, that is
${\mathfrak K}_{{\bm i}} ({\bm X})= {\mathcal K}({\bm X}^{{\bm i}}) = {\mathbb E}[X_1^{i_1} \cdots X_n^{i_n}].$
If $m=n > 1$ and $({\bm \lambda}_1, \ldots, {\bm \lambda}_n)=({\bm e}_1, \ldots, {\bm e}_n),$ is the standard basis  of
${\mathbb R}^n,$ then ${\bm X}^{{\bm e}_j}=X_j$ for $j=1, \ldots,n $ and  \eqref{gencummultp} returns the ${\bm i}$-th multivariate cumulant \eqref{multisetcum10}.
\begin{example}
If ${\bm X}=(X_1,X_2,X_3)$ then ${\mathfrak K}^{1}_{122}({\bm X})={\mathbb E}[X_1 X_2^2 X_3^2],$  and ${\mathfrak K}^{1| 2| 2}_{100; 010; 001} ({\bm X})= \kappa_{122}({\bm X})$ the multivariate cumulant of order $(1,2,2).$
\end{example}
The case $r_1=\ldots=r_m=1$ and ${\bm i}=\vec{{\bm 1}}_n$ is the one of special interest for our purposes. Indeed a multi-index partition $\Lambda \vdash \vec{{\bm 1}}_n$ with no equal columns corresponds to a partition $\pi \in \Pi_n,$  as the non-zero entries in each ${\bm \lambda}_j$ identify the row indexes belonging to the $j$-th block $B_j$ of $\pi.$ In such a case \eqref{gencummultp} reduces to $
 {\mathcal K} \left({\bm X}^{{\bm \lambda}_1}, \ldots,{\bm X}^{{\bm \lambda}_m}  \right) =  \mathfrak{K}(\pi),$ the generalized cumulant  \eqref{gencum}.

In order to provide a formula for expressing generalized multivariate cumulants in terms of multivariate cumulants, we need to represent complementary set partitions using multi-index partitions. This is the focus of the next subsection.

\subsection{Complementary $\vec{{\bm 1}}_n$-partitions}
Complementary set partitions can be characterized using partitions of 
$\vec{{\bm 1}}_n.$  To this aim we fix a different canonical representation ({\it cr2}) of partitions:  elements within blocks are ordered increasingly, and the blocks are arranged in lexicographic order. 
\begin{defn} \label{3.1}
The $\vec{{\bm 1}}_n$-partition corresponding to $\pi =B_1 |\ldots |B_m \in \Pi_{n,m}$ is the multi-index partition $\Lambda_{\pi} = ({\bm \lambda}_1, \ldots, {\bm \lambda}_m) \vdash \vec{{\bm 1}}_n$ such that its $(t,j)$-th element $({\bm \lambda}_j)_t = 1$ if $t \in B_j,$ otherwise $0.$
\end{defn}
The number $n$ of rows and the number $m=|\pi|$ of columns of $\Lambda$ are said the order and the size of the $\vec{{\bm 1}}_n$-partition respectively. We use the notation $\Lambda_{\pi} = {\bm \lambda}_1^{\intercal}|  \ldots| {\bm \lambda}_m^{\intercal}$ when it's needed.

Let $V_{\pi} = {\rm span} ({\bm \lambda}_1, {\bm \lambda}_2, \ldots)$ denote the subspace spanned by  the columns of ${\Lambda}_{\pi}$ with ${\rm dim}(V_{\pi}) = l(\Lambda_{\pi}).$ We say that $V_{\pi}$ is the column span of $\Lambda_{\pi}.$  
In particular $V_{{\mathbf 1}_n}$ denotes the subspace spanned by $\vec{{\bm 1}}_n.$
If $\Lambda_{\pi} =  \Lambda_{\tilde{\pi}}$ or every column of $\Lambda_{\tilde{\pi}}$ is a non-zero linear combination of columns of $\Lambda_{\pi}$ with coefficients in $\{0,1\}$ then $V_{\tilde{\pi}}={\rm span} (\tilde{{\bm \lambda}}_1, \tilde{{\bm \lambda}}_2, \ldots)\subseteq  {\rm span} ({\bm \lambda}_1, {\bm \lambda}_2, \ldots)=V_{\pi}.$  If we consider the set inclusion $\subseteq$ as partial order, then the greatest lower bound of two  column spans of  different $\vec{{\bm 1}}_n$-partitions is given by their intersection. 
\begin{example} \label{3.1}
We have  $\Lambda_{{\bm 1}_n} =\vec{{\bm 1}}_n^{\,\intercal}$ and  $\Lambda_{\vec{{\bm 0}}_n} = I_n,$ where  $I_n$ denotes the identity matrix. The following are $\vec{{\bm 1}}_4$-partitions:
$${\Lambda}_{1|234} = 1000|0111, \,  {\Lambda}_{123|4}=1110|0001,
\, {\Lambda}_{12|34}=1100|0011, \,  \Lambda_{12|3|4}=1100|0010|0001.$$ 
Notice that $V_{12|34} \subseteq V_{12|3|4}$ so that $V_{12|34} \cap V_{12|3|4} = V_{12|34}.$  Instead $V_{123|4} \cap V_{12|34} = V_{{\mathbf 1}_4},$  since no column of ${\Lambda}_{123|4}$ is a sum of columns from ${\Lambda}_{12|34}$ and viceversa. 
\end{example}
\begin{defn}
Two $\vec{{\bm 1}}_n$-partitions  $\Lambda$ and $\tilde{\Lambda}$ are said to be complementary if $V \cap \tilde{V} = V_{{\mathbf 1}_n},$ with  $V,\tilde{V}$  column spans of $\Lambda, \tilde{\Lambda}$ respectively.
\end{defn}

\begin{thm} \label{mainth}
$\pi \lor \tilde{\pi} = \pi^{\star}$  if and only if  $V_{\pi} \cap V_{\tilde{\pi}} = V_{\pi^{\star}}.$  
\end{thm}
\begin{proof}
Let $\Lambda_{\pi} = ({\bm \lambda}_1, {\bm \lambda}_2, \ldots)$ and $\Lambda_{\tilde{\pi}} = (\tilde{{\bm \lambda}}_1, \tilde{{\bm \lambda}}_2, \ldots)$ the $\vec{{\bm 1}}_n$-partitions 
corresponding to $\pi, \tilde{\pi} \in \Pi_n$ respectively. Notice that $\pi \leq \tilde{\pi}$ if and only if $V_{\tilde{\pi}} \subseteq V_{\pi}.$ Indeed $\Lambda_{\pi} =  \Lambda_{\tilde{\pi}}$ or every column of $\Lambda_{\tilde{\pi}}$ is a non-zero linear combination of columns of $\Lambda_{\pi}$ with coefficients in $\{0,1\}.$ Thus the least upper bound $\pi \lor \tilde{\pi}$  corresponds to the greatest lower bound $V_{\pi} \cap V_{\tilde{\pi}}.$ 
\end{proof}
\begin{cor} \label{cor3.1}
Complementary set partitions correspond to complementary $\vec{{\bm 1}}_n$-partitions.
\end{cor}
The following proposition parallels the swapping property recalled in Section 2. 
\begin{prop}
Suppose $\pi_1, \pi_2 \in \Pi_{n,m}$ whose block cardinalities correspond to the same integer partition $\lambda \vdash n$ regardless of the order of the blocks. 
If there exists a permutation  $\sigma$ that reorders the elements within the blocks of 
$\pi_1$ and $\pi_2$ and 
$V_{{\pi}_1} \cap V_{{\pi}} = V_{{\mathbf 1}_n}$ then 
$V_{{\pi}_2} \cap V_{\tilde{\pi}} = V_{{\mathbf 1}_n},$   where $\tilde{\pi}$ is  obtained from ${\pi}$ by applying the permutation $\sigma$ to the elements within its blocks.
\end{prop}
\begin{proof}
If there exists a permutation $\sigma$ reordering the elements in the blocks of $\pi_1$ and $\pi_2,$ then there exist square permutation matrices $P$ and $Q$ of $n$ and $m$ rows respectively such that $\Lambda_ {\pi_1} = P \Lambda_{\pi_2} Q.$ Thus  $\Lambda_{\pi_1} {\bm x}^{\intercal} = \Lambda_{\pi} {\bm x}^{\intercal}$ is equivalent to  $\Lambda_{\pi_2} {\bm y}^{\intercal} = \Lambda_{\tilde{\pi}} {\bm y}^{\intercal}$ where  $\Lambda_{\tilde{\pi}} = P \Lambda_{{\pi}} Q. $ Indeed as $P$ and $Q$ are permutation matrices, then $P^2 = I_n$ and $Q^2=I_m$ and
$\Lambda_{\pi_1} {\bm x}^{\intercal} = \Lambda_{\pi} {\bm x}^{\intercal} \Leftrightarrow  P \Lambda_{\pi_2} Q {\bm x}^{\intercal} = P P \Lambda_{\pi} Q Q {\bm x}^{\intercal} \Leftrightarrow P \Lambda_{\pi_2} Q {\bm x}^{\intercal} = P \Lambda_{\tilde{\pi}} Q {\bm x}^{\intercal} \Leftrightarrow
\Lambda_{\pi_2} {\bm y}^{\intercal} = \Lambda_{\tilde{\pi}} {\bm y}^{\intercal}$
where ${\bm y}^{\intercal} = Q {\bm x}^{\intercal}.$ Therefore, since 
 $V_{{\pi}_1} \cap V_{{\pi}} = V_{{\mathbf 1}_n}$ by hypothesis,   $V_{{\pi}_2} \cap V_{\tilde{\pi}} = V_{{\mathbf 1}_n}$ follows.
\end{proof}
\begin{example} \label{ex3.4}
In Example \ref{swapex}, $\pi=1|234$ is obtained from $\tilde{\pi}=123|4$ by swapping $1$ and $4.$  Both have $(3,1) \vdash 4$ as block cardinalities, regardless of the order of the blocks, and 
$${\Lambda}_{1|234} =\begin{pmatrix} 1 & 0 \\ 0 & 1 \\ 0 & 1\\ 0 & 1 \end{pmatrix}  = P {\Lambda}_{123|4} Q =  \begin{pmatrix} 0 & 0 & 0 & 1 \\  0 & 1 & 0 & 0 \\  0 & 0 & 1 & 0  \\ 1 & 0 & 0 & 0 \end{pmatrix} \begin{pmatrix} 1 & 0 \\ 1 & 0 \\ 1 & 0\\ 0 & 1 \end{pmatrix} \begin{pmatrix} 0 & 1 \\ 1 & 0 \end{pmatrix}. $$
\end{example}
Intersection matrices of set partitions can be recovered by multiplying the corresponding $\vec{{\bm 1}}_n$-partitions with suitable permutation matrices, as stated in the following proposition.
\begin{prop}
If $\pi_1 \in \Pi_{n,k}$ and $\pi_2 \in \Pi_{n,l},$  then
$\pi_1 \cap \pi_2 = P_1 ({\Lambda}_{\pi_1}^{\!\!\intercal}{\Lambda}_{\pi_2})P_2,$ where $\pi_1 \cap \pi_2$ is the $[k \times l]$ intersection matrix of Definition \ref{intmat} and $P_1, P_2$ are permutation matrices of dimensions
$[ k \times k ]$ and $[ l \times l ]$ respectively. 
\end{prop}
\begin{proof} Since ${\Lambda}_{\pi_1}$ is a $[n \times k]$ matrix and 
${\Lambda}_{\pi_2}$ is a $[n \times l]$ matrix, then ${\Lambda}_{\pi_1}^{\!\!\intercal} {\Lambda}_{\pi_2}$ is a $[k \times l]$ matrix such that  the $(i,j)$-th entry gives the number of coupled $1'$s between the $i$-th column of ${\Lambda}_{\pi_1},$ corresponding to the $i$-th block of $\pi_1,$ and the $j$-th column of ${\Lambda}_{\pi_2},$ corresponding to the $j$-th block of $\pi_2.$ If the canonical representations  {\it cr1} and {\it cr2} of $\pi_1$ and $\pi_2$ agree, the $(i,j)$-th entry of ${\Lambda}_{\pi_1}^{\!\!\intercal}{\Lambda}_{\pi_2}$ gives  the $(i,j)$-th entry of $\pi_1 \cap \pi_2,$ and $P_1 = I_k, P_2 =I_l$  identity matrices. Otherwise, there exits a permutation of blocks between $\pi_1$ and $\pi_2.$ Therefore $\pi_1 \cap \pi_2$ differs from ${\Lambda}_{\pi_1}^{\!\!\intercal}{\Lambda}_{\pi_2}$ by suitable permutations of rows and/or columns, which are encoded by left multiplication with a suitable permutation matrix $P_1$ (possibly equal to $I_k$) and right multiplication with a suitable permutation matrix $P_2$  (possibly equal to $I_l$.) 
\end{proof}
As a corollary, the equivalence classes of complementary set partitions can be constructed using products of $\vec{{\bm 1}}_n$-partitions. Indeed, given ${\pi}$, the partitions $\pi_1$ and $\pi_2$ are equivalent if ${\Lambda}_{\tilde{\pi}}^{\!\!\intercal} {\Lambda}_{\pi_1}$ is equal to ${\Lambda}_{\tilde{\pi}}^{\!\!\intercal} {\Lambda}_{\pi_2}$ after suitably permuting their rows and/or columns.
\begin{example}
The partitions $\pi_1=134|2$ and $\pi_2=234|1$ are in the same equivalence class with representative $\pi=124|3$ since
\begin{equation}
\pi_1 \cap \pi = \begin{pmatrix} 
2 & 1 \\ 
1 & 0 
\end{pmatrix} = \pi_2 \cap \pi = 
\Lambda_{\pi_1}^{\!\!\intercal} \Lambda_{\pi}  = \begin{pmatrix} 
0 & 1 \\ 
1 & 0 
\end{pmatrix} \Lambda_{\pi_2}^{\!\!\intercal} \Lambda_{\pi} 
\begin{pmatrix} 
1 & 0 \\ 
0 & 1 
\end{pmatrix}.
\end{equation}
\end{example}  
\subparagraph{The algorithm {\tt NullSpace}: {\tt csp\_NS}.} According to Corollary
\ref{cor3.1}, given two $\vec{{\bm 1}}_n$-partitions $\Lambda_{\pi}, \Lambda_{\tilde{\pi}},$ one way to determine if they are complementary is to compute a basis of  $V_{\pi} \cap V_{\tilde{\pi}}$ and check if this basis reduces to $\vec{{\bm 1}}_n.$ A basis for  $V_{\pi} \cap V_{\tilde{\pi}}$ can be  computed using the nullspace\footnote{The nullspace of a matrix $A$ is the set of solutions to the equation $A {\bm x} = {\mathbf 0}.$} of the block matrix $[{\Lambda}_{\pi} | - {\Lambda}_{\tilde{\pi}}].$ Therefore, given a $\vec{{\bm 1}}_n$-partition $\Lambda_{\pi},$ to enumerate all its complementary $\Lambda_{\tilde{\pi}}$ it is necessary to find all pairs $(\Lambda_{\pi}, \Lambda_{\tilde{\pi}})$ such that  the nullspace  of $[{\Lambda}_{\pi} | - {\Lambda}_{\tilde{\pi}}]$ is  $V_{{\mathbf 1}_n}.$ Because computing a nullspace numerically is time-intensive, various strategies can be employed to speed up the overall process. For example, some $\vec{{\bm 1}}_n$-partitions can be excluded initially, as they are clearly not complementary to  $\Lambda_{\pi}.$ This is the case of $\Lambda_{{\mathbf 0}_n},$ since $V_{\pi} \subseteq V_{{\mathbf 0}_n} = {\mathbb R}^n,$ but it also applies to all $\vec{{\bm 1}}_n$-partitions $\Lambda_{\tilde{\pi}}$ such that:  $\Lambda_{\tilde{\pi}}$ and  $\Lambda_{\pi}$ share at least one column;   at least one column of $\Lambda_{\tilde{\pi}}$ can be expressed  as linear combination of some columns of $\Lambda_{\pi}$ or viceversa;   ${\rm dim}(V_{\pi}) + 
{\rm dim}(V_{\tilde{\pi}}) \geq n + k$ with $k=2, \ldots,n$ as 
${\rm dim}(V_{\pi} \cap V_{\tilde{\pi}})=k > 1.$ 
While using $\vec{{\bm 1}}_n$-partitions and subspaces simplifies the search for complementary $\vec{{\bm 1}}_n$-partitions to a nullspace calculation, the algorithm faces significant challenges in terms of computational time (see Table 1, algorithm {\tt csp\_NS}). 
\subparagraph{Generalized cumulant.}
The generalized cumulant of ${\bm X} = (X_1, \ldots, X_n)$ of order $n$ and size $m,$ corresponding to  the $\vec{{\bm 1}}_n$-partition $\Lambda_{\pi} = ({\bm \lambda}_1, \ldots, {\bm \lambda}_m )$  is 
\begin{equation}
{\mathfrak K}_{\Lambda_{\pi}}({\bm X}) = {\mathfrak K}_{{\bm \lambda}^{\intercal}_1; \ldots; {\bm \lambda}^{\intercal}_m}({\bm X}) = {\mathcal K}( {\bm X}^{{\bm \lambda}_1}, \ldots,  {\bm X}^{{\bm \lambda}_m}).
\label{gencum1}
\end{equation}
Notice that $\mathfrak{K}_{\Lambda_{\pi}}({\bm X}) = \mathfrak{K}({\pi})$ in \eqref{gencum}. Moreover,  $\mathfrak{K}_{\vec{{\bm 1}}^{\intercal}_n}({\bm X})$ corresponds to the joint moment in \eqref{jointM} and $\mathfrak{K}_{{\bm e}_1^{\intercal}; \ldots, {\bm e}_n^{\intercal}}({\bm X})$ corresponds to the joint cumulant in \eqref{jointcum}, with $({\bm e}_1, \ldots, {\bm e}_n)$ the standard basis  of ${\mathbb R}^n.$ If $\Lambda_{\pi} = ({\bm \lambda}_1, \ldots, {\bm \lambda}_m) \vdash 
\vec{{\bm 1}}_n$ the expression of the generalized cumulant in terms of joint moments  is
\begin{equation}
{\mathfrak K}_{\Lambda_{\pi}}({\bm X}) = \sum_{\Lambda^{\star} \in {\mathcal C}_{\pi}} 
\kappa_{\bm X}(\Lambda^{\star})
\label{gcum2}
\end{equation}
where  ${\mathcal C}_{\pi} = \{\Lambda^{\star} = (\boldsymbol{\lambda}_1^{\star},\boldsymbol{\lambda}_2^{\star},\ldots) \vdash \vec{{\bm 1}}_n | V_{\pi^{\star}} \cap V_{\pi} = V_{{\bm 1}_n}  \}$ with  $V_{\pi^{\star}}, V_{\pi}$ column spans of $\Lambda^{\star},\Lambda_{\pi}$ respectively and $\kappa_{\bm X}(\Lambda^{\star}) = \kappa_{\boldsymbol{\lambda}_1^{\star}}({\bm X}) \kappa_{\boldsymbol{\lambda}_2^{\star}}({\bm X}) \cdots.$ and  Formula \eqref{gcum2} parallels formula \eqref{gencumincumul}. The proof is given in the Appendix.

\subsection{Labelling rules and transformations of $\vec{{\bm 1}}_n$-partitions}
In the following, using complementary $\vec{{\bm 1}}_n$-partitions, we provide a formula to express generalized multivariate cumulants using multivariate cumulants. To this aim, different subscripts are assigned to repeated r.v.'s so that eq. \eqref{gcum2} can be applied. We then show how to recover the initial information on repeated r.v.'s  in order to convert the joint moments in \eqref{gcum2} into multivariate moments. This method is suitable for implementation in non-symbolic platforms like {\tt R}.
\begin{example} \label{gen1}
For the computation of ${\rm cov}(X_1, X_2^2),$ consider 
${\rm cov}(Y_1, Y_2 Y_3) = {\mathfrak K}_{\scriptscriptstyle{100 ; 011}} = \kappa_{\scriptscriptstyle{111}} + \kappa_{\scriptscriptstyle{101}} \kappa_{\scriptscriptstyle{010}}  + \kappa_{\scriptscriptstyle{110}} \kappa_{\scriptscriptstyle{001}}.$
To obtain ${\rm cov}(X_1, X_2^2)$ consider the cumulant expansion of ${\rm cov}(Y_1,$ $Y_2 Y_3) $ where the last two binary indices are summed. This yields ${\rm cov}(X_1,X_2^2)  = \kappa_{\scriptscriptstyle{12}} + 2 \kappa_{\scriptscriptstyle{11}} \kappa_{\scriptscriptstyle{01}}.$
\end{example}
\noindent 
To implement this strategy, we define appropriate mappings between multi-index partitions and  $\vec{{\bm 1}}_n$-partitions. In the following let ${\mathcal L}_{{\bm i}}$ denote  the set of all multi-index partitions $\Lambda =  ({\bm \lambda}_1^{r_1}, \ldots, {\bm \lambda}_m^{r_m}) \vdash {\bm i} \in {\mathbb N}_0^n$ and let ${\mathcal M}_{|{\bm i}|}$ denote the set of all $\vec{{\bm 1}}_{|{\bm i}|}$-partitions $\Lambda_{\pi}$  with $|{\bm i}|=r_1 |{\bm \lambda}_1| + \ldots + r_m |{\bm \lambda}_m|.$ 
\begin{defn}[Canonical transformation of a multi-index partition] \label{6.2}
The canonical transformation $\varphi: {\mathcal L}_{{\bm i}} \mapsto  {\mathcal M}_{|{\bm i}|}$ maps  a multi-index partition $\Lambda \in {\mathcal L}_{{\bm i}}$ to a $\vec{{\bm 1}}_{|{\bm i}|}$-partition $\Lambda_{\pi} \in {\mathcal M}_{|{\bm i}|}$ 
where for $i=1, \ldots, r_1$
$$({\bm s}_i)_j = \left\{ \begin{array}{ll} 
1 & j=(i-1) |{\bm \lambda}_1|+ 1, \ldots, i |{\bm \lambda}_1| \\
0 & {\rm otherwise}
\end{array}\right.$$
and for $q=2, \ldots, m, t=1, \ldots, r_{q}$ and $m_{q,t}=r_1 |{\bm \lambda}_1| + \cdots + r_{q-1} |{\bm \lambda}_{q-1}| + (t-1) |{\bm \lambda}_{q}|$
$$({\bm s}_{r_1+\cdots+r_{q-1}+t})_j = \left\{ \begin{array}{ll} 
1 & j=m_{q,t}+1, \ldots, m_{q,t}+ |{\bm \lambda}_{q}| \\
0 & {\rm otherwise}.
\end{array}\right.$$
\end{defn}
\begin{example}\label{6.5}
Consider $M=\{1^{(1)},2^{(2)},3^{(2)}\}$ and the subdivision $S= 1|2|2|33$ corresponding to the multi-index partition ${\Lambda} = ({\bm \lambda}_1, {\bm \lambda}_2^2,{\bm \lambda}_3) \vdash (1,2,2)$ with $ {\bm \lambda}_1=(1,0,0)^{\intercal}, {\bm \lambda}_2 = (0,1,0)^{\intercal}, {\bm \lambda}_3 = (0,0,2)^{\intercal}.$ The $\vec{{\bm 1}}_5$-partition $\Lambda_{1|2|3|45} = 10000|01000|00100|00011$ is the canonical transformation of $\Lambda$ through $\varphi.$  
\end{example}
To transform a $\vec{{\bm 1}}_{|{\bm i}|}$-partition  $\Lambda_{\pi}\in {\mathcal M}_{|{\bm i}|}$ into a multi-index partition  $\Lambda \in {\mathcal L}_{\bm i},$ a suitable grouping rule is required for the rows of $\Lambda_{\pi}$. 
\begin{defn} [Labeling rule] \label{6.3}
Given a multiset $M = \{1^{(i_1)}, \ldots,n^{(i_n)} \}$  the labeling rule induced by $M$ is defined as $\sigma_{{\bm i}}: [p] \mapsto [n]$ with $p = |{\bm i}| \geq n$  such that  $\sigma^{-1}_{\bm i}(k) = \{ t_{k-1}+1, \ldots, t_k\}$ for $k \in [n]$ with
$t_k = \sum_{j=1}^k i_j$ and  $t_0=0.$ 
\end{defn}
\begin{example}\label{6.4}
With reference to the multiset $M=\{1^{(1)},2^{(2)},3^{(2)}\}$ in Example \ref{6.5} the labeling rule  $\sigma_{(1,2,2)}: \{1,2,3,4,5\} \mapsto \{1,2,3\}$ is defined by $\sigma_{(1,2,2)}(1) =1, \sigma_{(1,2,2)}(2)=\sigma_{(1,2,2)}(3) = 2, \sigma_{(1,2,2)}(4) = \sigma_{(1,2,2)}(5) = 3.$
\end{example} 
A multi-index partition  ${\Lambda} \in {\mathcal L}_{\bm i}$ can be constructed from a 
$\vec{{\bm 1}}_{|{\bm i}|}$-partition $\Lambda_{{\pi}} \in {\mathcal M}_{|{\bm i}|}$
by applying a labeling rule $\sigma_{\bm i},$ as follows. 
\begin{defn} [Reverse transformation under a labeling map] \label{5.4}
The reverse transformation $\Phi_{{\sigma}_{\bm i}}(\Lambda_{{\pi}}): {\mathcal M}_{|{\bm i}|} \mapsto {\mathcal L}_{\bm i}$  under the labeling rule ${\sigma}_{\bm i}$ maps a $\vec{{\bm 1}}_{|{\bm i}|}$-partition $\Lambda_{{\pi}} = ({{\bm s}}_1, \ldots, {{\bm s}}_{m}) \in {\mathcal M}_{|{\bm i}|}$ 
into a  multi-index partition ${\Lambda}=({\bm \lambda}_1, \ldots, {\bm \lambda}_m) \in {\mathcal L}_{\bm i}$  where for each $q=1, \ldots, m$ and  $k=1, \ldots,n$ the entries of 
${\bm \lambda}_q$ are given by  $({{\bm \lambda}}_q)_k = \sum_{t \in \sigma_{\bm i}^{-1}(k)} ({{\bm s}}_q)_t.$  
\end{defn}
\begin{example}
Let $\Lambda_{13|24|5} = (\bm{s}_1, \bm{s}_2, \bm{s}_3) \in \mathcal{M}_5$ be a $\vec{\bm{1}}_5$-partition, where ${\bm s}_1 = (1,0,1,0,0)^{\intercal},$ $ {\bm s}_2=(0,1,0,1,0)^{\intercal}, {\bm s}_3=(0,0,0,0,1)^{\intercal}$. 
Apply the reverse transformation $\Phi_{{\sigma}_{\bm i}}$ with $\bm{i} = (1,2,2)$ using the labeling rule $\sigma_{(1,2,2)}$ of Example~\ref{6.4}. The result is 
$ {\Lambda} =\Phi_{\sigma_{(1,2,2)}}({\Lambda}_{13|24|5}) = ({\bm \lambda}_1, {\bm \lambda}_2, {\bm \lambda}_3) \in {\mathcal L}_{(1,2,2)}$ with  ${\bm \lambda}_1 = (1,1,0)^{\intercal}, {\bm \lambda}_2= (0,1,1)^{\intercal}, {\bm \lambda}_3 = (0,0,1)^{\intercal}.$ Indeed, each $\bm{\lambda}_q$ is computed by summing the components of $\bm{s}_q$ indexed by $\sigma_{\bm{i}}^{-1}(k)$, that is the second and the third row of $\Lambda_{13|24|5}$ corresponding to  $\sigma_{\bm i}^{-1}(2)=\{2,3\}$   as well as the fourth and the fifth row corresponding to $\sigma_{\bm i}^{-1}(3)=\{4,5 \}.$ Distinct $\vec{\bm{1}}_5$-partitions can be mapped to the same multi-index partition $\Lambda \in \mathcal{L}_{(1,2,2)}$ via the reverse transformation $\Phi_{\sigma_{(1,2,2)}}$. Examples include:  $\Lambda_{12|34|5} = 11000|00110|00001, \Lambda_{12|35|4}=11000|00101|00010,$ and $ \Lambda_{13|25|4} = 10100|01001|00010.$ 
\end{example}
 Using the labeling rule $\sigma_{{\bm i}},$  a partition $\pi = B_1| \ldots| B_m \in \Pi_{|{\bm i}|,m}$ induces a subdivision $S=M_1|M_2|\ldots|M_m$ of $M=\{1^{(i_1)}, \ldots, n^{(i_n)}\}$ as follows: for each $j \in [p]$, where $p = |\bm{i}|$, and each $q \in [m]$, the label $\sigma_{\bm{i}}(j) \in \bar{M}_q$ if $j \in B_q$.
By Definition \ref{5.4},  since $|\sigma^{-1}_{\bm i}(k)|=i_k,$ the integer $\sum_{t \in \sigma_{\bm i}^{-1}(k)} ({{\bm s}}_q)_t$ counts the number of occurrences of the element $k \in [n]$  within $M_q$ corresponding to $B_q.$   
\begin{lemma} \label{6.1}
If $\Lambda = ({\bm \lambda}_1^{r_1}, \ldots, {\bm \lambda}_m^{r_m}) \in {\mathcal L}_{\bm i}$ then $|\Phi^{-1}_{\sigma_{\bm i}}(\Lambda)| = d_{\Lambda},$
where $\Phi_{\sigma_{\bm i}}^{-1}({{\Lambda}}) = \{ \Lambda_{{\pi}} \in {\mathcal M}_{|{\bm i}|} | \Phi_{\sigma_{\bm i}} (\Lambda_{{\pi}})= \Lambda \}$ and  $d_{\Lambda}$ is given in \eqref{multfunct}.
\end{lemma}
\begin{proof}
Let  $\Lambda = ({\bm \lambda}_1, \ldots, {\bm \lambda}_{l(\Lambda)}) \vdash {\bm i}$ with
${\bm \lambda}_1 \geq \ldots \geq {\bm \lambda}_{l(\lambda)}$ and 
denote with $S$ the subdivision corresponding to $\Lambda.$ Note that the $\vec{{\bm 1}}_{|{\bm i}|}$-partitions in  $\Phi_{\bm i}^{-1}({{\Lambda}})$  encode all the partitions ${\pi}$ such that  if the integers in the blocks of ${\pi}$ are replaced by their images under  $\sigma_{\bm i},$ then the subdivision $S$ is recovered.  Thus
\begin{equation}
a_k = \frac{i_k!}{\prod_{q=1}^{l(\Lambda)}  [({\bm{\lambda}}_q)_k]!} \quad k=1, \ldots,m
\label{jgroup1}
\end{equation}
is the number of ways to split the $i_k$ elements of $\sigma_{\bm i}^{-1}(k)$ among the blocks $B_1,\ldots, B_{l(\Lambda)}$ of $\pi$, knowing that  $({\bm \lambda}_q)_k$ elements in $B_q$ are mapped in $k.$ Using \eqref{jgroup1} and grouping together the equal columns of $\Lambda,$ the product $\prod_{k=1}^n a_k$ reads 
\begin{equation}
\prod_{k=1}^n a_k = \frac{{\bm i}!}{({\bm \lambda}_1!)^{r_1} \cdots ({\bm \lambda}_m!)^{r_m}}.
\label{jgroup2}
\end{equation}
Recall that the blocks $B_1,\ldots, B_{l(\Lambda)}$ correspond to the multisets $M_1,\ldots, M_{l(\Lambda)}$  
of $S$ when the  $p=|{\bm i}|$ integers in $B_1,\ldots, B_{l(\Lambda)}$ are replaced by their images under $\sigma_{\bm i}.$  The result follows by observing that when considering the number of ways to split all the elements in $\{1^{(i_1)}, \ldots, n^{(i_n)}\}$ among the multisets of $S$, the product in \eqref{jgroup2} is overcounted by the permutations of equal multisets, whose number is $r_1! \cdots r_m!$. Indeed, among the multisets of $S$, there may be copies, and once a multiset is filled, its copies are automatically determined.
\end{proof}
The following theorem gives generalized multivariate cumulant in terms of multivariate cumulants.
\begin{thm} \label{thm4.2}Let $\Lambda =({\bm \lambda}_1^{r_1}, \ldots, {\bm \lambda}_m^{r_m}) \vdash {\bm i} \in \mathbb{N}_0^n$, and let $\Lambda_{\pi} = \varphi(\Lambda) \in {\mathcal M}_{|{\bm i}|}$ be the canonical transformation of $\Lambda$. The generalized multivariate cumulant associated with $\Lambda$ satisfies
\begin{equation}
{\mathfrak K}_{\Lambda}({\bm X}) = {\mathfrak K}^{r_1| \ldots| r_m}_{{\bm \lambda}^{\!\intercal}_1; \ldots; {\bm \lambda}^{\!\intercal}_m}({\bm X}) 
= \sum_{\widetilde{\Lambda} = (\tilde{{\bm \lambda}}^{c_1}_1,\tilde{{\bm \lambda}}^{c_2}_2, \ldots ) \in {\mathfrak C}_{\Lambda}}  
a_{\tilde{\Lambda}} \, [\kappa_{\tilde{{\bm \lambda}}_1}({\bm X})]^{c_1}  [\kappa_{\tilde{{\bm \lambda}}_2}({\bm X})]^{c_2} \cdots
\label{genmulcum}
\end{equation}
where ${\mathfrak C}_{\Lambda} =\{ \tilde{\Lambda} \vdash {\bm i} \mid \Phi_{\sigma_{\bm i}}({\Lambda}_{\tilde{\pi}}) = \tilde{\Lambda} \text{ for } \Lambda_{\tilde{\pi}} \in {\mathcal C}_{\pi} \}$, ${\mathcal C}_{\pi}$ is the set of all complementary ${\bm 1}_{|{\bm i}|}$-partitions of $\Lambda_{\pi}$, $\Phi_{\sigma_{\bm i}}$ is the reverse transformation under $\sigma_{\bm i}$, and
$a_{\tilde{\Lambda}} = \left| \Phi_{\sigma_{\bm i}}^{-1}({\tilde{\Lambda}}) \cap {\mathcal C}_{\pi} \right|.$
\end{thm}
\begin{proof}
Set $p=|{\bm i}|$ and define ${\bm Y}=(Y_1, Y_2, \ldots, Y_{p})$ by
\begin{equation}
 Y_j = X_{\sigma_{\bm i}(j)},  \quad \hbox{\rm for} \,\, j=1,2,\ldots,{p}
\label{27}
\end{equation}
 where  $\sigma_{{\bm i}}$ is the labelling rule in Definition \ref{5.4}. Suppose $\Lambda_{\pi} = \varphi(\Lambda) = ({\bm s}_1, \ldots, {\bm s}_{l(\Lambda)}) \in {\mathcal M}_{p}.$   Then ${\mathfrak K}_{\Lambda}({\bm X}) =
{\mathfrak K}_{\varphi(\Lambda)}({\bm Y})$ and from \eqref{gcum2}
\begin{equation} 
 {\mathfrak K}_{{\bm s}^{\intercal}_1; \ldots; {\bm s}^{\intercal}_{l(\Lambda)}}({\bm Y}) = \sum_{{\Lambda}^{\star} \in {\mathcal C}_{\pi}} \kappa_{{\bm Y}}({\Lambda}^{\star}) = \sum_{{\Lambda}^{\star} \in {\mathcal C}_{\pi}} \kappa_{{\bm \lambda}^{\star}_1}({\bm Y})\cdots \kappa_{{\bm \lambda}^{\star}_{l({\Lambda}^{\star})}}({\bm Y})
\label{gencumY}
\end{equation}
where  ${\mathcal C}_{\pi}$ is the set of all complementary ${\bm 1}_{p}$-partitions of $\Lambda_{\pi} =\varphi(\Lambda),$ that is $ {\mathcal C}_{\pi} = \{ {\Lambda}^{\star}= ( {\bm \lambda}^{\star}_1, \ldots, {\bm \lambda}^{\star}_{l({\Lambda^{\star}})}) \vdash {\bm 1}_{p} | V_{\varphi(\Lambda)} \cap {V}^{\star}  = V_{{\bm 1}_p}  \}$ with  ${V}^{\star}$ and $V_{\varphi(\Lambda)} $  the column spans of $\Lambda^{\star}$ and 
$\Lambda_{\pi}=\varphi(\Lambda)$ respectively.  Now, consider for example
$\kappa_{{\bm \lambda}^{\star}_1}({\bm Y})$ in \eqref{gencumY}. Observe that
$$ \kappa_{{\bm \lambda}^{\star}_1}({\bm Y})=  {\mathcal K}(\underbrace{X_{\sigma_{\bm i}(1)}}_{{\bm \lambda}^{\star}_{11}}, \underbrace{X_{\sigma_{\bm i}(2)}}_{{\bm \lambda}^{\star}_{12}}, \ldots, \underbrace{X_{\sigma_{\bm i}(p)}}_{{\bm \lambda}^{\star}_{1p}}) = \kappa_{j_1, \ldots, j_n} ({\bm X})$$ 
where each index $j_t = \sum_{j \in \sigma^{-1}_{\bm i}(t)}  ({\bm \lambda}^{\star}_1)_j$ is the number of subscripts $j$ of the entries of ${\bm Y}$ that are mapped in $t$ under $\sigma_{\bm i}$ for all $t=1, \ldots, n.$ Since 
$|\sigma^{-1}_{\bm i}(t)|=i_t$ then $j_t \leq i_t$ for all $t=1, \ldots,n.$
Thus we have
\begin{equation}
 \kappa_{{\bm Y}}({\Lambda}^{\star}) = \kappa_{{\bm \lambda}^{\star}_1}({\bm Y})\cdots \kappa_{{\bm \lambda}^{\star}_{l({\Lambda}^{\star})}}({\bm Y}) = \underbrace{\kappa_{j_1, \ldots, j_n} ({\bm X}) \cdots  \kappa_{t_1, \ldots, t_n} ({\bm X})}_{l({\Lambda}^{\star})} =  \kappa_{{\bm X}}(\tilde{\Lambda}) 
\label{maineq}
\end{equation}
where \(\tilde{\Lambda} = \Phi_{\sigma_{\bm i}}(\Lambda^{\star})\) is the image of \(\Lambda^{\star}\) under the reverse transformation \(\Phi_{\bm i}\), with columns \((j_1, \ldots, j_n), \ldots, (t_1, \ldots, t_n)\) and \(l(\tilde{\Lambda}) = l(\Lambda^{\star})\). 
Note that  \eqref{maineq} holds for every $\vec{\bm 1}_{p}\)-partition \(\Lambda^{\star} \in \Phi_{\bm i}^{-1}(\tilde{\Lambda})$ that is complementary to
$\Lambda_{\pi}.$ Thus the number of such partitions is $a_{\tilde{\Lambda}} = \left| \Phi_{\sigma_{\bm i}}^{-1}({\tilde{\Lambda}}) \cap {\mathcal C}_{\pi} \right|.$
Therefore, the result follows by grouping the terms on the right-hand side of \eqref{gencumY} according to $\tilde{\Lambda} \in \mathfrak{C}_{\Lambda}$, and multiplying each $\kappa_{\bm X}(\tilde{\Lambda})$ by its corresponding coefficient 
$a_{\tilde{\Lambda}}.$
\end{proof}
\section{Applications} \label{Appl}
\subparagraph{Computational results.}
To demonstrate that the proposed algorithm based on two-block partitions is the most efficient, we implemented all the described methods for generating complementary set partitions within a single software environment. Since Stafford’s algorithm requires symbolic computation, we selected {\tt Maple 2024} \cite{maple2024} as the common platform for all implementations. The computational times\footnote{All routines are available upon request. Benchmarks were run on a laptop with an 11th Gen Intel(R) Core(TM) i7-1165G7 @ 2.80GHz.} are reported in Table 1. In particular, in Table 1, the first column lists the integer partition corresponding to the cardinalities of the blocks of $\pi$, which is the only required input due to the swapping property. The second column presents the results of the new method proposed in Section~\ref{newal}. The third column shows the performance of Stafford’s algorithm. The fourth and fifth columns report the results of the algorithms based on connected graphs, using the {\tt IsConnected} function and the Laplacian matrix, respectively. The final column corresponds to the algorithm that computes the nullspace of the intersection between column spans of $\vec{\bm 1}_n$-partitions. 
\begin{table}[H]
\centering
\begin{tabular}{ |c||c|c|c|c|c|   }
 \hline
 integer partition& {\tt csp} & {\tt csp\_AS} & {\tt csp\_GR} & {\tt csp\_LA} & {\tt csp\_NS} \\
\hline\hline
 $(1,1,2,2) \vdash 6$     &  0.0       & 0.0       &  0.016   &  0.094     & 0.187 \\
 $(2,2,2) \vdash 6$        &  0.0       & 0.016   &  0.016   &  0.172     & 0.296 \\
 $(2,2,3) \vdash 7$        & 0.016    & 0.031    &  0.282   &  1.234     & 1.516 \\
 $(3,4) \vdash 7$           & 0.0       & 0.031    &  0.234   &  1.313     & 1.437 \\
$(1,1,2,2,2) \vdash 8$   &  0.062    & 0.313   &  0.750   &  3.187     & 3.828 \\
$(1,3,4) \vdash 8$         &  0.016   & 0.328   &  0.906    &  5.516     & 6.078 \\
$(1,2,2,4) \vdash 9$       &  0.313   & 1.375  &  5.062    &  31.688   & 33.891 \\
$(2,3,4) \vdash 9$         &  0.110    & 1.172  &  6.234    &  40.125   & 43.719 \\
 $(2,2,2,2,2) \vdash 10$ & 1.218    & 9.110   &  39.140  &  241.782  & 659.093 \\
 $(2,2,3,3) \vdash 10$    & 0.718    & 8.609   &  41.907  &  278.593  & 876.39 \\
 \hline
\end{tabular}\caption{Computation times of the algorithms (implemented in {\tt Maple 2024}) that generate complementary set partitions for a given partition $\pi$.}
\label{table:1}
\end{table}
According to the results in Table 1, the algorithms based on nullspaces and Laplacian matrices are quite similar in performance for small values of $n.$ The algorithm using the {\tt Maple} {\tt IsConnected} function is faster than the previous two but slower than Stafford's algorithm. On the other hand, the algorithm based on two-block partitions not only demonstrates significantly lower computational times compared to all the other algorithms but also can be easily implemented on non-symbolic platforms. In fact, it only requires a routine to generate multi-index partitions. In {\tt R}, this routine is already available in the {\tt kStatistics} package\footnote{The routine is currently available upon request from the authors but will be included in the {\tt kStatistics} package shortly.} \cite{kStatistics}. The implementation strategy for {\tt R} is described below.
\subparagraph{The csp algorithm in {\tt R}.} In {\tt R}, ${\vec{\bm 1}}_n$ partitions offer an efficient way to handle set partitions. As a result, a version of the {\tt csp algorithm} outlined in Section \ref{newal} has been developed, using complementary ${\vec{\bm 1}}_n$ partitions.
Suppose we need to list the elements of ${\mathcal C}_{\pi}$ in \eqref{gcum2}, the set of complementary ${\vec{\bm 1}}_n$-partitions of $\Lambda_{\pi} = ({\bm \lambda}_1, \ldots, {\bm \lambda}_m)$. The steps of the newly proposed method can be summarized as follows:
\begin{description}
\item[{\it a)}]  suppose  $t_1, \ldots, t_j| t_{j+1},  \ldots, t_m$ a partition  in two blocks of the subscripts of  $({\bm \lambda}_1, \ldots, {\bm \lambda}_m),$  and  set ${\bm v}_{1}={\bm \lambda}_{t_1} + \cdots + {\bm \lambda}_{t_j}$ and ${\bm v}_{2}= {\bm \lambda}_{t_{j+1}} + \cdots + {\bm \lambda}_{t_m};$
\item[{\it b)}] for all the $\vec{{\bm 1}}_n$-partitions $\Lambda^{(1)} \vdash {\bm v}_{1}$ and all the $\vec{{\bm 1}}_n$-partitions $\Lambda^{(2)} \vdash {\bm v}_{2},$  construct the block matrix $\Lambda^{\star} = (\Lambda^{(1)},\Lambda^{(2)});$ 
\item[{\it c)}] repeat steps {\it a)} and {\it b)} for all partitions in two blocks of the subscripts of  $({\bm \lambda}_1, \ldots, {\bm \lambda}_m),$ and 
denote with ${\mathcal T}^{(n)}_{\pi,m}$ the set of $\vec{{\bm 1}}_n$-partitions $\Lambda^{\star}=(\Lambda^{(1)},\Lambda^{(2)})$ constructed as above; 
\item[{\it d)}]  set ${\mathcal C}_{\pi} = {\mathcal M}_n -  {\mathcal T}^{(n)}_{\pi,m},$ where ${\mathcal M}_{n}$ denotes the set of all $\vec{{\bm 1}}_n$-partitions. 
\end{description}
\subparagraph{Generalized multivariate cumulants using  the {\tt csp} algorithm.}
From \eqref{gencumincumul} and Theorem \ref{thm3.2} generalized cumulants can be recovered as
\begin{equation}
\mathfrak{K}(\pi) = \sum_{\tilde{\pi} \in {\Pi}_n} \prod_{B \in \tilde{\pi}}  \kappa(B) 
- \sum_{\tilde{\pi} \in {\mathcal T}^{(n)}_{\pi,m}} \prod_{B \in \tilde{\pi}}  \kappa(B) 
\label{gencumincumul1}
\end{equation}
where ${\mathcal T}^{(n)}_{\pi,m}$ is given in \eqref{fund3}. A similar strategy is applied to recover generalized multivariate cumulants \eqref{genmulcum} as explained below. From \eqref{gencumY}, the computation of generalized multivariate cumulants
$\mathfrak{K}_{\Lambda}({\bm X})$ can be performed using \eqref{gencumincumul1}, that is
\begin{equation}
{\mathfrak K}_{\varphi(\Lambda)}({\bm Y}) = \sum_{{\Lambda}^{\star} \in {\mathcal M}_{|{\bm i}|}} \kappa_{{\bm Y}} (\Lambda^{\star})  - \sum_{{\Lambda}^{\star} \in {\mathcal T}^{(|{\bm i}|)}_{\pi,m}} \kappa_{{\bm Y}} (\Lambda^{\star}) 
\label{gencumY21}
\end{equation}
where ${\bm Y}$ is given in \eqref{27}, ${\mathcal T}^{(|{\bm i}|)}_{\pi,m}$ is the set of all not complementary 
$\vec{{\bm 1}}_{|{\bm i}|}$-partitions of $\Lambda_{\pi} = \varphi(\Lambda)$ and $m = l(\Lambda_{\pi}).$ By repeating the same arguments of the proof of Theorem
\ref{thm4.2} and using Lemma \ref{6.1}, the first sum in \eqref{gencumY21} is 
\begin{equation}
\sum_{{\Lambda}^{\star} \in {\mathcal M}_{|{\bm i}|}} \kappa_{{\bm Y}} (\Lambda^{\star}) =\sum_{\widetilde{\Lambda} \vdash {\bm i}}  d_{\tilde{\Lambda}} \kappa_{{\bm X}} (\tilde{\Lambda}).
\label{firstadd}
\end{equation}
Notice that the computation of the rhs  of \eqref{firstadd} is generally faster than the one at the lhs. For the latter sum in \eqref{gencumY21}, with similar arguments we have
\begin{equation}
\sum_{{\Lambda}^{\star} \in {\mathcal T}^{(|{\bm i}|)}_{\pi,m}} \kappa_{{\bm Y}} (\Lambda^{\star})  =   \sum_{\widetilde{\Lambda}  \in \Phi_{{\bm i}}[{\mathcal T}^{(|{\bm i}|)}_{\pi,m}]}   b_{\tilde{\Lambda}} \kappa_{{\bm X}} (\tilde{\Lambda})
\label{genmulcum1}
\end{equation}
where $b_{\tilde{\Lambda}} = |\Phi_{{\bm i}}^{-1}({\tilde{\Lambda}})   \cap {\mathcal NC}_{\pi}|$ and ${\mathcal NC}_{\pi}$ denotes the set of all not complementary ${\bm 1}_{|{\bm i}|}$-partitions of $\Lambda_{\pi} =\varphi(\Lambda).$ Suppose $\Lambda_{\pi} = (\bm{\lambda}_1, \ldots, {\bm \lambda}_m).$  To compute the rhs of \eqref{genmulcum1} efficiently, we use \eqref{fund3ter}, which gives
$$\sum_{\widetilde{\Lambda}  \in \Phi_{{\bm i}}[{\mathcal T}^{(|{\bm i}|)}_{\pi,m}]}   b_{\tilde{\Lambda}} \kappa_{{\bm X}} (\tilde{\Lambda}) = \sum_{k=1}^l (-1)^{k+1} \sum_{1 \leq i_1 \leq \cdots \leq i_k \leq |{\bm i}|} \sum_{\Lambda^{\star} \in \cap_{j=1}^k {\mathcal T}_{i_j}} \kappa_{{\bm X}}(\Phi_{\bm i}(\Lambda^{\star})),$$
where ${\mathcal T}_{i_j}$ is the set analogous to the one given in \eqref{fund3bis} but for $\vec{\bm 1}_{|{\bm i}|}$-partitions, that is the set of all  $\Lambda^{\star} \in {\mathcal M}_{|i|} $
such that $\Lambda^{\star} = (\Lambda^{(1)}, \Lambda^{(2)})$ with $\Lambda^{(j)} \vdash {\bm v}_j = \sum_{i \in B_j}  {\bm \lambda}_i$ for $j=1,2$ and for a fixed $C_1|C_2 \in \Pi_{m,2}.$

\subparagraph{Estimation of generalized multivariate cumulants.}
An unbiased estimator $\hat{\mathfrak K}_{\Lambda}(\bm X)$ of the generalized multivariate cumulant ${\mathfrak K}_{\Lambda}(\bm X)$ is obtained by replacing 
the products of multivariate cumulants appearing on the rhs of \eqref{genmulcum}  with their corresponding unbiased estimators. Specifically for each $\tilde{\Lambda}=({\bm \lambda}_1^{c_1},  {\bm \lambda}_2^{c_2}, \ldots) \in {\mathfrak C}_{\Lambda}$   the products appearing on  the rhs of \eqref{genmulcum} are replaced by $k_{\tilde{\Lambda}}({\bm X})$ defined as 
\begin{equation}
{\mathbb E}[k_{\tilde{\Lambda}}({\bm X})] = {\mathbb E}\bigg[k_{\underbrace{\tilde{{\bm \lambda}}_1; \ldots, \tilde{{\bm \lambda}}_1}_{c_1}; \underbrace{\tilde{{\bm \lambda}}_2; \ldots, \tilde{{\bm \lambda}}_2}_{c_2}; \ldots} ({\bm X}) \bigg] = [\kappa_{\tilde{{\bm \lambda}}_1}({\bm X})]^{c_1}  [\kappa_{\tilde{{\bm \lambda}}_2}({\bm X})]^{c_2}\cdots.
\label{31}
\end{equation}
These statistics   are the so-called multivariate polykays \cite{di2008unifying}  typically expressed in terms of power sum symmetric functions given by $S_{t_1, \ldots, t_n} = S_{t_1, \ldots, t_n}({\bm X})=\sum_{l=1}^N X_{1,l}^{t_1} X_{2,l}^{t_2} \cdots X_{n,l}^{t_n}$ where the sum runs over $N$ istances of ${\bm X}=(X_1, X_2, \ldots, X_n).$
\begin{example} \label{exp11} Suppose $n=3.$ The unbiased estimator of 
$\kappa_{\scriptscriptstyle{110}}({\bm X}) \kappa_{\scriptscriptstyle{001}}({\bm X})$ is the polykay
{\small $${k_{\scriptscriptstyle{110;001}}({\bm X}) \!\! = \!\! \frac{-S_{0,0,1} S_{1,0,0} S_{0,1,0}+\left(N -1\right) S_{0,0,1} S_{1,1,0}+S_{0,1,0} S_{1,0,1}+S_{0,1,1} S_{1,0,0}-N S_{1,1,1}}{N \left(N -1\right) \left(N -2\right)}}.
$$} 
\end{example} 
Implementing this strategy to recover estimators of ${\mathfrak K}_{\bm X}(\Lambda)$ is computationally intensive. Instead, a more efficient procedure\footnote{Routines for computing unbiased estimators of generalized multivariate cumulants are also available upon request.} for obtaining  $\hat{\mathfrak K}_{\Lambda}(\bm X)$ consists in applying  the canonical transformation $\varphi$ of
$\Lambda.$ Indeed from \eqref{gencumY}, we have 
\begin{equation}
\hat{\mathfrak K}_{\Lambda}(\bm X) = \hat{\mathfrak K}_{\varphi(\Lambda)}({\bm Y})
\label{rid}
\end{equation}
 with ${\bm Y}$ given in \eqref{27} and $\hat{\mathfrak K}_{\varphi(\Lambda)}({\bm Y})$ unbiased estimator of the generalized cumulant ${\mathfrak K}_{\varphi(\Lambda)}({\bm Y}).$ In turn, the expression of $\hat{\mathfrak K}_{\varphi(\Lambda)}({\bm Y})$ involves power sum symmetric functions of the form
\begin{equation}
S_{t_1, \ldots, t_p}({\bm Y})=\sum_{l=1}^N Y_{1,l}^{t_1} Y_{2,l}^{t_2} \cdots Y_{p,l}^{t_n}
\label{psY}
\end{equation}
with $(t_1, \ldots, t_p)$ binary vectors. The expression of $\hat{\mathfrak K}_{\Lambda}(\bm X)$ in terms of power sums in ${\bm X}$ is then recovered by \eqref{27}, that is plugging $X_{\sigma_{\bm i}(j)}$ in place of
$Y_j$ in \eqref{psY}. Moreover the binary vector  $(t_1, \ldots, t_p)$ is transformed into 
the multi-index $(j_1, \ldots, j_n)$ where $j_s = |\sigma^{-1}_{\bm i}(s)|, s=1, \ldots, n.$ 
\begin{example} \label{5.8}
In Example \ref{gen1} we have ${\rm cov}(X_1, X_2^2) = \kappa_{1,2} + 2 \kappa_{1,1} \kappa_{0,1}.$ Its unbiased estimator is
$$\hat{\mathfrak K}^{1| 1}_{10; 02}({\bm X}) = k_{12}({\bm X}) + 2 k_{11;01} ({\bm X}) = \frac{N S_{12}({\bm X}) - S_{10}({\bm X}) S_{02}({\bm X})}{N(N-1)}$$ which can be recovered from
\begin{equation}
\hat{\mathfrak K}_{100;011}({\bm Y}) =  \frac{N S_{111}({\bm Y})  - S_{100}({\bm Y})  S_{011}({\bm Y}) }{N(N-1)}
\label{polY}
\end{equation}
with $S_{111}({\bm Y}) = \sum_{i} Y_{1,i} Y_{2,i} Y_{3,i}$ replaced by $S_{12} ({\bm X})= \sum_{i} X_{1,i} X_{2,i}^2$
and $S_{100}({\bm Y}) = \sum_{i} Y_{1,i}, $ $S_{011}({\bm Y}) = \sum_{i} Y_{2,i} Y_{3,i}$ replaced by $S_{10}({\bm X}) = \sum_{i} X_{1,i}, S_{02}({\bm X}) =  \sum_{i} X_{2,i}^2$ respectively.  
\end{example}
From the previous discussion, the expression of $\hat{\mathfrak K}_{\Lambda}(\bm X)$ is recovered from an unbiased estimator of the generalized cumulant ${\mathfrak K}_{\varphi(\Lambda)}({\bm Y}).$  In turn, an unbiased estimator $\hat{\mathfrak K}_{\Lambda_{\pi}}({\bm Y})$  of the generalized cumulant ${\mathfrak K}_{\Lambda_{\pi}}({\bm Y})$ can be obtained by replacing the product of joint cumulants on the rhs of \eqref{gencumY} with the joint polykays $k_{{\bm \lambda}_1;\ldots; {\bm \lambda}_m}({\bm Y}),$ such that ${\mathbb E}[k_{{\bm \lambda}_1;\ldots; {\bm \lambda}_m}({\bm Y})] =\kappa_{{\bm \lambda}_1}({\bm Y}) \cdots \kappa_{{\bm \lambda}_m}({\bm Y})$ which is the product of joint cumulants of ${\bm Y},$ see Example \ref{exp11}. As before, this strategy is computationally intensive. A more efficient procedure consists in using a change of variables. Indeed $
\hat{\mathfrak K}_{{\bm \lambda}_1;\ldots; {\bm \lambda}_m}({\bm Y})  = \hat{\mathfrak K}_{{\bm 1}_m} (Z_1, \ldots, Z_m) $
with $Z_j={\bm Y}^{{\bm s}_j}$ for $j=1,\ldots,m.$ The expression of 
$\hat{\mathfrak K}_{{\bm 1}_m}(Z_1, \ldots, Z_m) $ involves  power sum symmetric functions of the form
\begin{equation}
S_{t_1, \ldots, t_m}({\bm Z})=\sum_{l=1}^N Z_{1,l}^{t_1} Z_{2,l}^{t_2} \cdots Z_{m,l}^{t_n}
\label{psY1}
\end{equation}
with $(t_1, \ldots, t_m)$ binary vectors. The expression of $\hat{\mathfrak K}_{{\bm \lambda}_1;\ldots; {\bm \lambda}_m}({\bm Y}) $ in terms of the power sums in ${\bm Y}$ is then recovered plugging ${\bm Y}^{{\bm s}_j}$ in place of $Z_j$ in \eqref{psY1} and observing that $(t_1, \ldots, t_m)$ is transformed into $(j_1, \ldots, j_p) = \sum_{j=1}^m t_j {\bm s}^{\intercal}_j.$ 
\begin{example} \label{exp3}
In Example \ref{5.8}, we have $\hat{\mathfrak K}_{\scriptscriptstyle{100 ; 011}}  ({\bm Y}) = \hat{\mathfrak K}_{1,1}(Z_1, Z_2)$ with $Z_1 = {\bm Y}^{(1,0,0)}=Y_1$ and $Z_2 = {\bm Y}^{(0,1,1)}=Y_2 Y_3.$ Since
\begin{equation}
\hat{\mathfrak K}_{1,1}(Z_1, Z_2) = \frac{N S_{1,1}({\bm Z}) - S_{1,0}({\bm Z}) S_{0,1}({\bm Z})}{N(N-1)} 
\label{poly22}
\end{equation}
with $S_{1,1}({\bm Z}) = \sum_{l=1}^N Z_{1,l} Z_{2,l}, S_{1,0}({\bm Z}) = \sum_{l=1}^N Z_{1,l} $ and $S_{0,1}({\bm Z}) = \sum_{l=1}^N  Z_{2,l},$ plugging $Y_{1,i}$ in place of $Z_{1,i}$ and  $Y_{2,i} Y_{3,i}$ in place of $Z_{2,i},$ we recover 
$S_{1,1}({\bm Z}) =S_{1,1,1}({\bm Y}),  S_{1,0}({\bm Z})  = S_{1,0,0}({\bm Y}),  S_{0,1}({\bm Z})  = S_{0,1,1}({\bm Y}).$  Replacing all these power sums in \eqref{poly22}, \eqref{polY} is obtained.
\end{example}
\section{Conclusions and open problems}
This paper introduces a new combinatorial method for enumerating all complementary set partitions of a given partition, demonstrating that only two-block partitions are sufficient to identify all not complementary partitions. This approach outperforms existing procedures in terms of computational efficiency and can be implemented in non-symbolic programming environments such as {\tt R}. In particular, classical methods that utilize connected graphs to generate complementary set partitions are discussed. We also consider the Stafford's algorithm, which automates the calculation of generalized cumulants and includes procedures for complementary set partitions. However, this method is limited by its reliance on symbolic algebra tools and the employment of a not open source software as {\tt Mathematica}. Finally, we have analyzed one more method
using the columns of $\vec{{\bm 1}}_n$-partitions (paralleling set partitions) as spanning bases of suitable subspaces. The complementary $\vec{{\bm 1}}_n$-partitions correspond to subspaces whose intersections are spanned by the unit vector. Comparisons of computational times were made using {\tt Maple 2024}.  

An extension of the notion of generalized cumulants to the case of repeated r.v.'s is proposed, based on multiset subdivisions. These indexes are intermediate between multivariate moments and multivariate cumulants.  A closed-form formula is derived to express these generalized multivariate cumulants in terms of  multivariate cumulants, together with a combinatorial interpretation of the coefficients involved. The key tool is the definitions of suitable functions transforming multi-index partitions, corresponding to multi-set subdivisions, in set partitions, corresponding to ${\vec{\bm 1}}_n$-partition and viceversa. The use of these functions and suitable dummy variables allow to develop very efficient procedures for estimating generalized multivariate cumulants through multivariate polykays.

Exploring the broader implications of all these methods in multivariate analysis, such as their potential for new estimation procedures, hypothesis testing, and models involving complex dependence structures, remains an open avenue for future investigation. 
For example, developing efficient algorithms for the computation of cumulants of polykays and their multivariate generalizations remains an ongoing challenge. At present the only routine available is in {\tt Mathstatica} \cite{mathstatica} for the calculation of the cumulants of $k$-statistics\footnote{Univariate polykays reduce to $k$-statistics when the product of joint cumulants is replaced by a single joint cumulant.}. Future research aims might include also the characterization of these families of generalized cumulants in the setting of random matrices \cite{McCulDinardo}.

%
\section{Appendix: generalized cumulants and ${\vec{\bm 1}}_n$-partitions}
Denote with ${\mathcal M}_n$ the set of all $\vec{{\bm 1}}_n$-partitions $\Lambda_{\pi}.$ 
\begin{prop} \label{17bis} $\mu_{\bm X}\!(\Lambda_{\pi}) = \sum_{\{ \Lambda \in {\mathcal M}_n | V_{\pi} \subseteq V \}}  \kappa_{\bm X}(\Lambda)$ with $V_{\pi}, V$ column spans of $\Lambda_{\pi}, \Lambda$ respectively. 
\end{prop}
\begin{proof}
Suppose $\Lambda_{\pi} = (\boldsymbol{\lambda}_1, \boldsymbol{\lambda}_2, \ldots)$. From the second identity in \eqref{multfunct}, it follows that
$$
\mu_{\bm X}\!(\Lambda_{\pi}) = \mu_{\boldsymbol{\lambda}_1}({\bm X}) \mu_{\boldsymbol{\lambda}_2}({\bm X}) \cdots = \sum_{\Lambda_1 \vdash \boldsymbol{\lambda}_1,\, \Lambda_2 \vdash \boldsymbol{\lambda}_2,\, \ldots} \kappa_{\bm X}(\Lambda_1) \kappa_{\bm X}(\Lambda_2) \cdots.
$$
Define $\Lambda = (\Lambda_1, \Lambda_2, \ldots)$. Observe that $\Lambda$ is a $\vec{\boldsymbol{1}}_n$-partition. In particular, each column $\boldsymbol{\lambda}_i$ of $\Lambda_{\pi}$ can be written as a non-zero linear combination of the columns of $\Lambda$ with coefficients in $\{0,1\}$ and therefore the column space $V_{\pi}$ of $\Lambda_{\pi}$ is contained in the column space $V$ of $\Lambda,$ i.e., $V_{\pi} \subseteq V.$ The result follows from the multiplicative property $\kappa(\Lambda) = \kappa(\Lambda_1) \kappa(\Lambda_2) \cdots.$
\end{proof}
The expression of the generalized cumulant in terms of joint moments is given as follows.
\begin{prop} \label{17ter}
If $\Lambda_{\pi}=({\bm \lambda}_1, \ldots, {\bm \lambda}_m) \vdash 
\vec{{\bm 1}}_n$ then 
\begin{equation}
{\mathfrak K}_{{\bm \lambda}^{\intercal}_1; \ldots; {\bm \lambda}^{\intercal}_m}({\bm X}) = \sum_{\{ \tilde{\Lambda} \in {\mathcal M}_n | \tilde{V} \subseteq V_{\pi} \}}  (-1)^{l(\tilde{\Lambda}) - 1} (l(\tilde{\Lambda})-1)!\mu_{\bm X}\!(\tilde{\Lambda})
\label{eqcc}
\end{equation}
where  $\tilde{V}, V_{\pi}$ are the column spans of $\tilde{\Lambda}, \Lambda_{\pi}$ respectively and $\mu_{\bm X}\!(\tilde{\Lambda})$ is the product of joint moments of ${\bm X}$ as defined in \eqref{multfunct}.
\end{prop}
\begin{proof}
In \eqref{gencum1} set  $Y_j= {\bm X}^{{\bm \lambda}_j}$ for $j=1,2,\ldots, m.$ From the second identity in \eqref{cummom1} we have
\begin{equation}
{\mathfrak K}_{{\bm \lambda}^{\intercal}_1; \ldots; {\bm \lambda}^{\intercal}_m}({\bm X}) = \kappa_{\scriptscriptstyle{\vec{{\bm 1}}_m}} (Y_1,\ldots,Y_m) = \sum_{\Lambda \vdash \vec{{\bm 1}}_m} (-1)^{l(\Lambda)-1}(l(\Lambda)-1)! \mu_{{\bm Y}}\!(\Lambda)
\label{gencum2}
\end{equation}
where the sum runs over all $\vec{{\bm 1}}_m$-partitions $\Lambda$ of order $m \leq n.$ In each term $\mu_{{\bm Y}}\!(\Lambda)$ plug ${\bm X}^{{\bm \lambda}_j}$ in place of  $Y_j$ for $j=1,2,\ldots,m.$  Assume without loss of generality that
$\Lambda = ({\bm \lambda}^{\star}_1, {\bm \lambda}^{\star}_2,\ldots,{\bm \lambda}^{\star}_l)$ with $l \leq m$. Then we can write
$
\mu_{{\bm Y}}\!(\Lambda) = \mu_{{\bm \lambda}^{\star}_1} (Y_1,\ldots,Y_m) \, \mu_{{\bm \lambda}^{\star}_2} (Y_1,\ldots,Y_m) \cdots$ $ \mu_{{\bm \lambda}^{\star}_l} (Y_1,\ldots,Y_m)
$
where each term takes the form $
\mu_{{\bm \lambda}^{\star}_j} (Y_1,\ldots,Y_m) = {\mathbb E}[Y_1^{\lambda^{\star}_{1j}} \cdots Y_m^{\lambda^{\star}_{mj}}] =  {\mathbb E}[{\bm X}^{ {\bm \lambda}_1 \lambda^{\star}_{1j} + \cdots +{\bm \lambda}_m \lambda^{\star}_{mj}}].
$
Define
$
\tilde{\bm \lambda}_j = {\bm \lambda}_1  \lambda^{\star}_{1j} + \cdots +{\bm \lambda}_m  \lambda^{\star}_{mj}$ for $j=1,2,\ldots,l$ and notice that $ \tilde{\Lambda}=(\tilde{\bm \lambda}_1, \ldots, \tilde{\bm \lambda}_l) \in {\mathcal M}_n.$
Then $\mu_{\bm Y}\!(\Lambda)=\mu_{\bm X}\!(\tilde{\Lambda})$. The result follows by observing that, as $\Lambda$ varies in the set of all  $\vec{{\bm 1}}_m$-partitions, each $\tilde{\Lambda}$ is such that $ \tilde{V} \subseteq V_{\pi}$  since every column of $\tilde{\Lambda}$ is a non-zero linear combination of the columns of $\Lambda_{\pi}$ with coefficients in $\{0,1\}.$ 
\end{proof}
\begin{cor} \label{cor5.1}
If $\Lambda_{\pi}=({\bm \lambda}_1, {\bm \lambda}_2, \ldots) \vdash 
\vec{{\bm 1}}_n$ then
\begin{equation}
\sum_{{\{ \tilde{\Lambda} \in {\mathcal M}_n | \tilde{V} \subseteq V_{\pi} \}}} (-1)^{l(\tilde{\Lambda} ) - 1}(l(\tilde{\Lambda} )-1)! = 
\left\{\begin{array}{ll}
0, & \hbox{if $\Lambda_{\pi} \ne \vec{{\bm 1}}^{\intercal} _n$}\\
1, & \hbox{if $\Lambda_{\pi} = \vec{{\bm 1}}_n^{\intercal}$}
\end{array} \right.
\label{final}
\end{equation}
with $\tilde{V}, V_{\pi}$ column spans of $\tilde{\Lambda}, \Lambda_{\pi}$ respectively. 
\end{cor} 
\begin{proof}
If $\Lambda_{\pi} = \vec{{\bm 1}}^{\intercal} _n$ the sum on the lhs of \eqref{final} reduces to
 $(-1)^{l(\vec{{\bm 1}}^{\intercal} _n) - 1}(l(\vec{{\bm 1}}^{\intercal} _n )-1)! = 1.$  For $\Lambda_{\pi} \ne \vec{{\bm 1}}^{\intercal} _n$ recall that, using formal power series \cite{stanley2011enumerative}, the multi-indexed sequence $\{\kappa_{\bm i}\}$ in \eqref{cummom} is the formal cumulant sequence associated with a  multi-indexed sequence $\{\mu_{\bm i}\},$ with $\mu_{\bm 0} = 1,$ and not necessarily representing the moment sequence of any multivariate random vector ${\bm X}.$ In particular, if $\mu_{\bm i} = 1$ for all ${\bm i} \in  {\mathbb N}^n_0$ then from \eqref{cummom} it follows $\kappa_{\bm i} = 1$ if $|{\bm i}|=1$ otherwise $0.$
Therefore if $\Lambda_{\pi} \ne \vec{{\bm 1}}^{\intercal} _n,$ the rhs of \eqref{eqcc} vanishes corresponding to the value of the lhs of \eqref{final} when $\mu_{\bm X}\!(\Lambda)=1.$  
\end{proof}
\begin{cor} \label{3.3}  If $\Lambda_{\pi} = ({\bm \lambda}_1, \ldots, {\bm \lambda}_m) \vdash 
\vec{{\bm 1}}_n$ then
\begin{equation}
{\mathfrak K}_{{\bm \lambda}^{\intercal}_1; \ldots; {\bm \lambda}^{\intercal}_m}({\bm X})= \sum_{\Lambda^{\star} \in {\mathcal C}_{\pi}} 
\kappa_{\bm X}(\Lambda^{\star})
\label{gcum}
\end{equation}
where  ${\mathcal C}_{\pi} = \{\Lambda^{\star} \vdash \vec{{\bm 1}}_n | V_{\pi^{\star}} \cap V_{\pi} = V_{{\bm 1}_n}  \}$ with  $V_{\pi^{\star}}, V_{\pi}$ column spans of $\Lambda^{\star},\Lambda_{\pi}$ respectively. 
\end{cor}
\begin{proof}
From Proposition \ref{17bis} and  \ref{17ter}, we have
\begin{equation}
{\mathfrak K}_{{\bm \lambda}^{\intercal}_1; \ldots; {\bm \lambda}^{\intercal}_m}({\bm X}) =  \sum_{\{ \tilde{\Lambda} \in {\mathcal M}_n |  \tilde{V} \subseteq V_{\pi} \}} (-1)^{l(\tilde{\Lambda})-1}(l(\tilde{\Lambda})-1)! \sum_{\{ {\Lambda}^{\star} \in {\mathcal M}_n |  \tilde{V} \subseteq V^{\star} \}} \kappa_{\bm X}({\Lambda}^{\star})
\label{first}
\end{equation}
with  $V^{\star}, \tilde{V}$ column spans of $\Lambda^{\star}, \tilde{\Lambda}$ respectively. Grouping the outer sum in \eqref{first} according to the length of $\tilde{\Lambda}, $ the rhs of \eqref{first} becomes 
\begin{equation}
\sum_{{\Lambda}^{\star} \vdash \vec{{\bm 1}}^{\intercal}_n} \kappa_{\bm X}({\Lambda}^{\star}) -\sum_{\scriptscriptstyle{\tilde{\Lambda} \in {\mathcal I}_2}} \sum_{\scriptscriptstyle{{\Lambda}^{\star} \in {\mathcal S}_{2,\tilde{V}}}} \kappa_{\bm X}({\Lambda}^{\star}) +   \cdots + (-1)^{l-1}(l-1)!  \sum_{\scriptscriptstyle{\tilde{\Lambda} \in {\mathcal I}_l}} \sum_{\scriptscriptstyle{{\Lambda}^{\star} \in {\mathcal S}_{l,\tilde{V}}}} \kappa_{\bm X}({\Lambda}^{\star})
\label{terzo}
\end{equation}
where  ${\mathcal I}_i = \{ \tilde{\Lambda} \in {\mathcal M}_n | l(\tilde{\Lambda})=i \, \hbox{\rm and} \, \tilde{V} \subseteq V_{\pi}\}, {\mathcal S}_{i,\tilde{V}} = \{ {\Lambda}^{\star} \in {\mathcal M}_n |  \tilde{V} \subseteq V^{\star},\}$ with 
$\tilde{V}$ column span of $\tilde{\Lambda}$  for $i=2,\ldots,l$ with $l \leq m.$  Collecting together common $\kappa_{\bm X}({\Lambda}^{\star})$,  \eqref{terzo} returns
\begin{equation}
{\mathfrak K}_{{\bm \lambda}^{\intercal}_1; \ldots; {\bm \lambda}^{\intercal}_m}({\bm X})  =  \sum_{\Lambda^{\star} \vdash \vec{{\bm 1}}^{\intercal}_n} \kappa_{\bm X}(\Lambda^{\star}) \sum_{\{\tilde{\Lambda} \in {\mathcal M}_n | \tilde{V} \subseteq V_{\pi} \cap V^{\star} \}} (-1)^{l(\tilde{\Lambda}) - 1}(l(\tilde{\Lambda})-1)!
\label{quinto}
\end{equation}
with $\tilde{V}, V_{\pi}$ and $V^{\star}$ column spans of $\tilde{\Lambda},\Lambda_{\pi}$ and $\Lambda^{\star}$ respectively. 
 From Corollary \ref{cor5.1} the inner sum in \eqref{quinto} is always zero, except in the case  $V^{\star} \cap V_{{\pi}}= V_{{\bm 1}_n}.$ Indeed, under this condition, we have $\tilde{V} =  V_{{\bm 1}_n}$ and the inner sum in \eqref{quinto} to the single term $\tilde{\Lambda}=\vec{{\bm 1}}^{\intercal}_n$ yields the value 
$1$ by Corollary \ref{cor5.1}.  
\end{proof}
{\bf Acknowledgements.}
{\it Funding:} the research of E.D. was partially supported by the
MIUR-PRIN 2022 project “Non-Markovian dynamics and non-local equations”,
202277N5H9. 


\end{document}